\documentclass[11pt]{amsart}
\usepackage{amsmath,amsfonts,amssymb,amscd,amsthm,amsbsy,bbm,enumerate,color,graphicx,comment}
\usepackage[noadjust]{cite}
\IfFileExists{norbert}{\usepackage[colorlinks,citecolor=blue,linkcolor=blue]{hyperref}}{\usepackage[colorlinks,citecolor=blue,linkcolor=blue]{hyperref}}
\usepackage[usenames,dvipsnames]{xcolor}
\usepackage{url}

\usepackage{tikz}
\usepackage{caption}
\usepackage{subcaption}
\usepackage[utf8]{inputenc}
\usepackage{fullpage}

\numberwithin{equation}{section}

\newtheorem{thm}{Theorem}[section]
\newtheorem{lem}[thm]{Lemma}
\newtheorem{cor}[thm]{Corollary}
\newtheorem{prop}[thm]{Proposition}
\newtheorem{THM}{Theorem}

\theoremstyle{definition}
\newtheorem{ex}[thm]{Example}
\newtheorem{DEF}[thm]{Definition}

\theoremstyle{remark}   
\newtheorem{remark}[thm]{Remark}

\newcommand\B{{\mathcal B}}
\renewcommand\H{{\mathcal H}}
\newcommand\I{{\mathcal I}}
\newcommand\K{{\mathcal K}}
\renewcommand\P{{\mathcal P}}
\newcommand\R{{\mathbb R}}
\renewcommand\S{{\mathbb S}}
\newcommand\T{{\mathcal T}}
\newcommand{\N}{\mathbb{N}}

\newcommand\e{\varepsilon}
\newcommand\Om{\Omega}
\newcommand\gam{\gamma} 
\newcommand\lam{\lambda}
\newcommand\Lam{\Lambda}
\renewcommand\k{\kappa}
\newcommand{\n}{\vec{n}}
\newcommand{\cc}{\mathsf{c}}

\newcommand{\sd}{\operatorname{sd}}
\newcommand{\dist}{\operatorname{dist}}
\newcommand\per{\operatorname{Per}}
\newcommand\sign{\operatorname{sign}}
\newcommand{\dv}{\operatorname{div}}
\DeclareMathOperator{\co}{Cone}
\DeclareMathOperator{\spn}{span}
\DeclareMathOperator{\trace}{trace}

\renewcommand\liminf{\mathop{\lim\,\inf}\limits}%
\newcommand\liminfs{\mathop{\lim\,\inf{}_*}\limits}%
\renewcommand\limsup{\mathop{\lim\,\sup}\displaylimits}%
\newcommand\limsups{\mathop{\lim\,\sup{}^*}\limits}%
\newcommand\argmax{\mathop{\arg\,\max}\limits}%

\newcommand{\oB}{{\overline B}}
\newcommand{\oL}{{\overline L}}

\newcommand{\tu}{{\widetilde u}}
\newcommand{\tOm}{{\widetilde \Omega}}
\newcommand{\hu}{{\widehat u}}
\newcommand{\hOm}{{\widehat \Omega}}

\newcommand{\Rn}{{\mathbb R^N}}
\newcommand{\Rpz}{{[0,\infty)}}

\newcommand{\ot}{{\Omega_t}}
\newcommand{\ott}{{(\ot)_{t\geq0}}}
\newcommand{\oz}{{\Omega_0}}

\newcommand{\lh}{\lambda^h}
\newcommand{\Lh}{{\Lambda^{h}}}
\newcommand{\kp}{\k_\phi}

\definecolor{darkgreen}{rgb}{0,0.4,0}

\begin{document}
\title{On volume-preserving crystalline mean curvature flow}
\author{Inwon Kim}
\address{Department of Mathematics, UCLA, Los Angeles, USA}
\thanks{IK is supported by NSF DMS 1900804 and the Simons Foundation Fellowship. }
\email{ikim@math.ucla.edu}

\author{Dohyun Kwon}
\address{Department of Mathematics, University of Wisconsin-Madison, 480 Lincoln Dr., Madison, WI 53706, USA}
\email{dkwon7@wisc.edu}
\thanks{}
\date{}

\author{Norbert Po{\v{z}}{\'a}r}
\address{Faculty of Mathematics and Physics, Institute of Science and Engineering, Kanazawa University, Kakuma town, Kanazawa, Ishikawa 920-1192, Japan}
\email{npozar@se.kanazawa-u.ac.jp}
\thanks{NP is supported by JSPS KAKENHI Wakate Grant (No. 18K13440).}

\subjclass[2010]{}


\begin{abstract}
In this work we consider the global existence of volume-preserving crystalline curvature flow in a non-convex setting. We show that a natural geometric property, associated with reflection symmetries of the Wulff shape, is preserved with the flow. Using this geometric property, we address global existence and regularity of the flow for smooth anisotropies. For the non-smooth case we establish global existence results for the types of anisotropies known to be globally well-posed.
\end{abstract}
\maketitle

\section{Introduction}
\label{sec:intro}
The motion of sets by crystalline curvature arises from physical applications such as crystal growth \cite{Cahn} or in statistical physics \cite{Spohn}, where sets evolve to decrease their anisotropic perimeter. We consider the volume-preserving version of such motions.  More precisely, we consider a flow of sets $(\Omega_t)_{t\geq0} $ moving with the outward normal velocity $V$ given by
\begin{align}
\label{model}
\tag{M}
V=\psi(\n)(-\k_\phi+\lam) \quad\hbox{ on } \partial \Om_t.
\end{align}
Here   $\k_\phi$ and $\n$  each denote the anisotropic mean curvature and the outward unit normal of $\partial \Omega_t$.   The forcing term $\lam = \lambda(t)$, coupled with the solution, is the Lagrange multiplier enforcing the volume constraint $|\Omega_t| = |\Omega_0|$.  The functions $\psi, \phi: \R^N\to [0, \infty)$ 
are positively one-homogeneous (see \eqref{eqn:one-homogeneous}), convex, positive away from the origin, and are respectively denoting mobility and anisotropy in the system. We will also require that they are symmetric with respect to a number of reflections: this will allow the solutions of \eqref{model} to preserve a related geometric property that is central to our analysis as we discuss below.

\medskip


The anisotropic curvature $\kappa_\phi$ is formally the first variation of the anisotropic perimeter functional 
\begin{align*}
\per_{\phi}(\Om) := \int_{\partial \Omega} \phi(\n) \,d\H^{N-1}.
\end{align*}
 If $\phi$ is $C^2$ away form the origin, one can verify that $\kappa_\phi = \dv_{\partial \Omega} D\phi(\n)$, where $\dv_{\partial \Omega}$ is the surface divergence of the so-called Cahn-Hoffman vector field $D\phi(\n)$. If no regularity except convexity is assumed for $\phi$ and if the graph of $\phi$ has corners, $\kappa_\phi$ is called a {\it crystalline } mean curvature. If $\phi$ is piece-wise linear so that its sub-level sets are convex polytopes, it is often called a {\it purely crystalline} anisotropy. We refer to \cite{T78,Taylor91,AngenentGurtin89,AT95} for further discussion of this problem in the variational setting. See Figure~\ref{fig:square-flow} for an example of a solution of \eqref{model}.

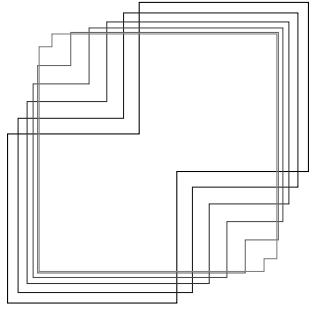
\begin{figure}
\centering
\begin{tikzpicture}
\draw[gray!0!black] (0.25,-0.25)--(2.0,-0.25)--(2.0,2.0)--(-0.25,2.0)--(-0.25,0.25)--(-2.0,0.25)--(-2.0,-2.0)--(0.25,-2.0) -- cycle;
\draw[gray!20!black] (0.458,-0.458)--(1.86,-0.458)--(1.86,1.86)--(-0.458,1.86)--(-0.458,0.458)--(-1.86,0.458)--(-1.86,-1.86)--(0.458,-1.86) -- cycle;
\draw[gray!40!black] (0.681,-0.681)--(1.74,-0.681)--(1.74,1.74)--(-0.681,1.74)--(-0.681,0.681)--(-1.74,0.681)--(-1.74,-1.74)--(0.681,-1.74) -- cycle;
\draw[gray!60!black] (0.916,-0.916)--(1.66,-0.916)--(1.66,1.66)--(-0.916,1.66)--(-0.916,0.916)--(-1.66,0.916)--(-1.66,-1.66)--(0.916,-1.66) -- cycle;
\draw[gray!80!black] (1.16,-1.16)--(1.6,-1.16)--(1.6,1.6)--(-1.16,1.6)--(-1.16,1.16)--(-1.6,1.16)--(-1.6,-1.6)--(1.16,-1.6) -- cycle;
\draw[gray!100!black] (1.41,-1.41)--(1.58,-1.41)--(1.58,1.58)--(-1.41,1.58)--(-1.41,1.41)--(-1.58,1.41)--(-1.58,-1.58)--(1.41,-1.58) -- cycle;
\end{tikzpicture}
\caption{Example of the volume-preserving flow in 2D with $\phi(p) = \psi(p) = \|p\|_1$ that converges to a rescaling of the square Wulff shape.}
\label{fig:square-flow}
\end{figure}

\medskip

Our goal in this paper is to establish a global existence for the flow \eqref{model} for sets starting from non-convex profiles. The difficulty lies in both the non-smoothness of $\phi$ and the low regularity of $\lambda$.  If the evolving surface is regular enough, then $\lambda$ can be explicitly written as the weighted average of the mean curvature over the surface,
\begin{align*}
\lambda(t) = \int_{\partial \Omega_t} \kappa_\phi \psi(\n) \,d\H^{N-1} / \int_{\partial \Omega_t} \psi(\n) \,d\H^{N-1},
\end{align*}
but there is no a priori regularity for $\lambda$ in general. In fact for non-convex Lipschitz domains, the forcing term $\lambda$ in \eqref{model} can be unbounded (see \cite[Example A.2]{KimKwo18b}).

\medskip

When $\phi$ is smooth and when the forcing $\lambda$ is a priori fixed as a bounded function in \eqref{model}, there have been both viscosity solution \cite{EvansSpruck,CCG91,Soner93,BarSonSou93} and variational approaches \cite{ATW93,AT95,LS95,cham04,cham08} to address the global well-posedness of the flow.

\medskip

It is well-known  that convexity is preserved in the flow \eqref{model}. The global-in-time existence of the convex flow, as well as exponential convergence to a rescaled version of the Wulff shape
\begin{align} 
\label{Wulff}
W_\phi := \{x \in \Rn: x \cdot p \leq \phi(p) \text{ for all }p \in \Rn\}
\end{align}
have been studied for both smooth \cite{Hui87, And01} and non-smooth $\phi$ \cite{BCCN09}. Beyond convex setting, there is no global-in-time existence result for \eqref{model} even for smooth anisotropies. To achieve this one likely needs to understand the pattern of topological changes that contributes to both the instability of the motion as well as the regularity of the forcing term $\lambda$. 
 In the isotropic case $\psi=\phi=|x|$, the global existence for \eqref{model} is proved under an energy convergence assumption that rules out abrupt topological changes: see \cite{LS95, MSS16, LaxSwa}. 
 
 \medskip
 
 Alternatively, one could also explore geometric conditions under which topological changes do not occur: this is the direction we pursue here. While it is suspected that star-shapedness is preserved in the evolution, it remains open to be proved even for the isotropic case.  For the isotropic case, \cite{KimKwo18a, KimKwo18b} introduced a stronger version of star-shapedness, called as a {\it reflection property}, motivated by \cite{FelKim14}.  Below we will introduce a geometric condition \eqref{eqn:1as} that naturally extends this property for anisotropic flows.

\medskip

To discuss the geometric property that \eqref{model} preserves, some notations are necessary. We represent the reflection symmetry of $\phi$ and $\psi$  in terms of the corresponding {\it root system} as follows. Let $\mathcal{P}$ be a finite {\it root system }(see \eqref{eqn:sym}) consisting of unit vectors in $\R^N$, with enough directions in the root system such that
\begin{align}
\label{as:ea}
\spn( \P \setminus \Pi ) = \Rn \hbox{ for any hyperplane $\Pi$ going through the origin}.
\end{align}

We say $\mathcal{P}$ is {\it compatible} with $\phi$ and $\psi$ if in addition $\phi$ and $\psi$ are invariant under reflection with respect to its elements, namely \eqref{eqn:rt}. The concept of the root system has been introduced in the context of reflection and Coxeter groups (see e.g. \cite{Hum90}). 
Examples of $\phi$ that allow a compatible root system are those whose Wulff shapes are convex regular polytopes in $\R^2$ and $\R^3$.

\medskip 

Given such a root system, we consider sets $\Omega$ for which there exists $\rho > 0$ such that 
\begin{equation}\label{volume}
|\Omega| > \K^N \rho^N|B_1(0)|,
\end{equation}
 satisfying the following reflection property:
\begin{equation}
\label{eqn:1as}
 \Psi_{\partial H} (\Omega) \cap H\subset \Omega \cap H \quad \hbox{ for any half-space $H \subset \Rn$ whose normal is in $\P$ and }  B_{\rho}(0) \subset H,
\end{equation}
where $\Psi_{\partial H}$ denotes the reflection operator with respect to the hyperplane $\partial H$; see Section~\ref{sec:geo}. Here the constant $\K = \K(\P) > 1$ is given in \eqref{eqn:k0}.  Note that the reflection property gets stronger as $\rho$ decreases since the family of eligible half-spaces grows: $W_{\phi}$ satisfies \eqref{eqn:1as} for $\rho=0$, making it the ideal shape for this property. 

\medskip

Our main observation is that the flow $\ott$ given by \eqref{model} preserves the property \eqref{volume}--\eqref{eqn:1as} for any root system satisfying \eqref{as:ea} compatible with $\phi$ and $\psi$. This geometric property in turn guarantees that $\Omega_t$ is a Lipschitz domain for each time (Theorem~\ref{thm:lpb}).  To incorporate $\lambda$ as a distribution, we will denote the coupled pair $(\ott, \Lambda)$ as a solution of \eqref{model}, where $\Lambda' = \lambda$ in the sense of distributions.

\begin{THM}
\label{THM:1}
Let $\phi\in C^2(\Rn \setminus \{0\})$, and let $\P$ be a root system compatible with $\phi$ and $\psi$. Then for any bounded open initial data $\Omega_0$ satisfying \emph{\eqref{volume}--\eqref{eqn:1as}}, there is a viscosity solution $(\ott,\Lam)$ of \eqref{model} starting from $\Omega_0$ that preserves volume and satisfies \eqref{eqn:1as} for all positive times.

Moreover, $\Lam\in C^{1/2}([0,\infty))$, and there exists a finite number of local neighborhoods $\{\mathcal{O}_i\}_{i=1}^n$ in $\Rn$ such that
\begin{itemize}
\item[(a)] $\bigcup_{i=1}^n \mathcal{O}_i \times [0,\infty)$  contains $\Gamma:= \bigcup_{t>0}(\partial\ot \times\{t\})$.
\item[(b)] For each  $i$,   $\Gamma$ restricted to $\mathcal{O}_i\times [0,\infty)$ can be represented as a graph of a function that is uniformly Lipschitz in space and uniformly H\"{o}lder continuous in time.
\end{itemize}
In addition, $\mathcal{O}_i$, the coordinates in $\mathcal{O}_i$ where the graph property holds, as well as their Lipschitz and H\"{o}lder constants, depend only on $\mathcal{P}$, the extremal values of $\phi$ and $\psi$ on $\S^{N-1}$, and $|\Omega_0|$.

\end{THM}

 See the end of Section~\ref{sec:gexi} for the proof of this Theorem.

\medskip

Additional challenges arise when $\phi$ is non-smooth. In this case the optimal regularity for the evolving set is only Lipschitz and $\kappa_\phi$ is understood as nonlocal. The flow develops flat features like faces and edges (see Figure~\ref{fig:square-flow}) that might break or bend during the evolution \cite{BNP99}. The well-posedness of the crystalline mean curvature flow has only recently been established, respectively with the level set formulation \cite{GigaGiga98, GigPoz20} and with variational approach that directly addresses the motion of the sets \cite{CMNP19, CMP17}; see Section~\ref{sec:global-existence} for further discussions.  The next theorem states that our results hold for the class of non-smooth anisotropies that were successfully addressed with these approaches.

\begin{THM}
\label{THM:2}  
If $\phi$ is {\it purely crystalline}, then the statements in Theorem~\ref{THM:1} hold as long as there is {\it no fattening} in the process of smooth approximations (see Section~\ref{sec:general-mobility}).

If $\psi$ is {\it $\phi$-regular}, then the statements in Theorem~\ref{THM:1} hold for a  {\it flow} $(\ott,\Lam)$ of \eqref{model} in the sense of  \cite{CMP17} (see Section~\ref{se:phi-regular-mobility}).
\end{THM}

\begin{remark} $\,$
\begin{itemize} 
\item[1.] Uniqueness of the flow \eqref{model} remains open, even in the isotropic case.
\item[2.]  Both conditions \eqref{volume} and \eqref{eqn:1as} are needed to ensure that $\ot$ is a Lipschitz domain. For instance, if $\P = \left\{ \pm e_i : 1 \leq i \leq N \right\}$, then we may have a domain with cusps satisfying \eqref{eqn:1as}. See also Remark~\ref{rem:pri} and Example~\ref{ex:sq}.
\item[3.] For smooth anisotropies, we expect that an approach similar to \cite[Section~5]{KimKwo18b} will lead to the asymptotic convergence of the flow to the Wulff shape. We do not pursue it here since the proof would at least require significant regularity analysis that deviates from the main topic of  the paper. For non-smooth anisotropies the asymptotic convergence to a Wulff shape remains open. 
\end{itemize}
\end{remark}

We now give two examples of sets that satisfy our geometric assumptions but that are not convex.
\begin{ex}
\label{ex:two-cubes}
Let
\begin{align*}
\phi(\xi) = \psi(\xi) := \|\xi\|_\infty := \max\{ | \xi \cdot e_i | : 1 \leq i \leq N \} 
\end{align*}
where $\{e_i\}_{i =1}^N$ is the standard basis of $\Rn$, and let $\P$ be given as 
\begin{align}
\label{eqn:sq}
\P = \left\{ \pm e_i : 1 \leq i \leq N \right\} \cup \{ \pm\tfrac{1}{\sqrt{2}}(e_i + e_j), \pm\tfrac{1}{\sqrt{2}}(e_i - e_j) : 1 \leq i < j \leq N \},
\end{align}
see Figure~\ref{fig:ex1}(a).
Then $\P$ satisfies \eqref{as:ea}, and the following union of two cubes with sufficiently large $C>0$
\begin{align*}
\Om_0:= \left([-C,C]^N - e_1 \right) \cup \left([-C-1,C+1]^N + e_1 \right)
\end{align*}
satisfies \eqref{eqn:1as} for $\rho=1$; see Figure~\ref{fig:ex1}(b).
\end{ex}
\begin{figure}
\begin{tikzpicture}
\begin{scope}
\foreach \t in {0, 45, ..., 315} {
\draw[->] (0,0) -- (\t:1);
}
\draw (0,-2.) node {(a)};
\draw (30:1.2) node {$\P$};
\end{scope}
\begin{scope}[scale=0.2,xshift=30cm]
\draw[->] (-8,0) -- (8,0);
\draw[->] (0, -8) -- (0,8);
\draw[thick] (-5,6) -- (7,6) node[above right] {$\Omega_0$} -- (7,-6) -- (-5,-6) -- (-5,-5) -- (-6,-5)
-- (-6,5) -- (-5,5) -- cycle;
\draw (0,0) circle [radius=1];
\draw (1,0) node[above right=-2pt] {$B_1(0)$};
\def\x{0.707}
\draw[very thin,dashed] (-6.5+\x,6.5+\x) -- (6.5+\x,-6.5+\x) 
(-6.5-\x,6.5-\x) -- (6.5-\x,-6.5-\x); 
\draw (0,-10) node {(b)};
\end{scope}
\end{tikzpicture}
\centering
\caption{The root system $\P$ and $\Omega_0$ in Example~\ref{ex:two-cubes} for $N = 2$.}
\label{fig:ex1}
\end{figure}
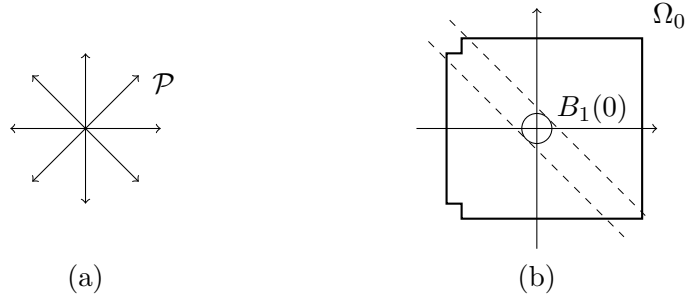

%

\begin{ex}
\label{ex:star}
In this example we consider $\phi$ and $\psi$ with a triangular symmetry.
For $N=2$, let
\begin{align*}
\phi(\xi) = \psi(\xi) := \max\{ \xi \cdot \eta_1, \xi \cdot \eta_2, \xi \cdot \eta_3 \} 
\end{align*}
where $\eta_1 = (\frac{\sqrt{3}}{2},\frac{1}{2})$, $\eta_2 = (-\frac{\sqrt{3}}{2},\frac{1}{2})$ and $\eta_3 = (0, -1)$. 
Then, $\phi$ and $\psi$ are invariant with respect to reflections given by elements of $\P := \left\{ \pm\left(\frac{1}{2}, -\frac{\sqrt{3}}{2}\right), \pm\left(\frac{1}{2}, \frac{\sqrt{3}}{2}\right), \pm(1,0) \right\}$ as in \eqref{eqn:rt}. In addition, $\P $ satisfies \eqref{as:ea}. In this case, the two equilateral triangles
\begin{align*}
\{ \xi : \phi(\xi) \leq 1\} \hbox{ and } \{ \xi : \phi(-\xi) \leq 1\},
\end{align*}
satisfy \eqref{eqn:1as} for all $\rho>0$. In particular, their union $\Omega_0$ also satisfies \eqref{eqn:1as} for all $\rho>0$; see Figure~\ref{fig:ex2}(b).
\end{ex}
\begin{figure}
\begin{tikzpicture}
\begin{scope}
\foreach \t in {0, 60, ..., 300} {
\draw[->] (0,0) -- (\t:1);
}
\draw (0,-2.) node {(a)};
\draw (30:1.2) node {$\P$};
\end{scope}
\begin{scope}[xshift=6cm]
\draw[->] (-1.5,0) -- (1.5,0);
\draw[->] (0, -1.5) -- (0,1.5);
\draw[thick] (0:0.5774) -- (30:1)
-- (60:0.5774) node[above right] {$\Omega_0$} -- (90:1)
-- (120:0.5774) -- (150:1)
-- (180:0.5774) -- (210:1)
-- (240:0.5774) -- (270:1)
-- (300:0.5774) -- (330:1)
-- cycle;
\draw (0, -2) node {(b)};
\end{scope}
\end{tikzpicture}
\centering
\caption{The root system $\P$ and $\Omega_0$ in Example~\ref{ex:star} for $N = 2$.}
\label{fig:ex2}
\end{figure}
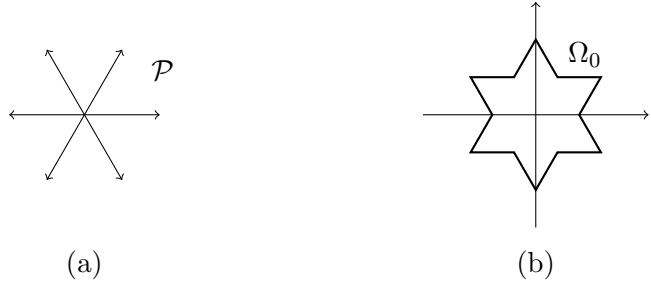

\medskip

\textbf{Outline of the paper.}
In Section~\ref{sec:geo} we study geometric properties of sets that satisfy \eqref{eqn:1as}. In Section~\ref{sec:geo1} we show that such sets are Lipschitz domains, by establishing interior and exterior cone properties at the boundary points. Since the reflection property is limited  to the directions in $\P$, the argument for this step is considerably more subtle than the one in \cite{KimKwo18a, KimKwo18b}.  Section~\ref{sec:geo2} provides a lower bound on in-radius of the sets satisfying \eqref{eqn:1as} in terms of its diameter.

In Section~\ref{sec:not}, we introduce a notion of viscosity solutions for \eqref{model}. While interested in the geometry of sets, we will adopt the level set approach, since it allows flexibility in perturbation arguments in our analysis.  We extend the notions developed in  \cite{KimKwo18b} to accommodate $\lambda$ that is only a distributional derivative of a continuous function. This is necessary due to the unknown regularity of the volume-preserving $\lambda$.

Sections~\ref{sec:pre}, \ref{sec:lam} and~\ref{sec:gexi} concern curvature flows with smooth anisotropy $\phi$. In Section~\ref{sec:pre} we show the preservation of the reflection property \eqref{eqn:1as} for level sets of viscosity solutions with fixed forcing. In Section~\ref{sec:lam}, a discrete-time scheme is introduced to approximate \eqref{model} with flows with piece-wise constant forcing. Due to the results from the previous sections, one can show that the discrete solutions have locally Lipschitz interfaces and thus a fattening phenomenon does not occur in their limit. As a consequence, we  prove Theorem~\ref{THM:1} in Section~\ref{sec:gexi}.

Lastly in Section~\ref{sec:global-existence} we address \eqref{model} with non-smooth $\phi$ and prove Theorem~\ref{THM:2}. Here the global well-posedness of \eqref{model} is established in terms of available notions of crystalline flow.

\section{Root system and geometric properties}
\label{sec:geo}

In this section we study geometry of sets that satisfy \eqref{eqn:1as}. For a unit vector $p\in \Rn$ we define the hyperplane 
$$
\Pi_{p}(s) : = \{ x \in \Rn : x \cdot p = s \},
$$ which divides $\Rn$ into the half-spaces 
$$\Pi^+_{p}(s) : = \{ x \in \Rn : x \cdot p > s \} \hbox{  and  } \Pi^-_{p}(s) : = \{ x \in \Rn : x \cdot p < s \}.
$$ Then the corresponding reflection map and the reflection property \eqref{eqn:1as} can be written respectively as  $\Psi_{\Pi_p(s)}(x)= x- 2  (x \cdot p - s) p$
and  
\begin{align}
\label{eqn:rp}
\Psi_{\Pi_p(s)} (\Om) \cap \Pi^-_{p}(s) \subset \Om \cap \Pi^-_{p}(s) \quad \hbox{ for all $p \in \P$ and $s  >\rho$.}
\end{align}

\medskip

Let us next review the notion of root systems, which are used to describe reflection symmetries of objects \cite[Section~1.2]{Hum90}.  A {\it root system} $\P$ in $\Rn$ is a set of nonzero vectors in $\Rn$ satisfying 
\begin{align}
\label{eqn:sym}
\P \cap (\R p) = \{p, -p\} \hbox{ and } \Psi_p \P = \P \hbox{ for all } p \in \P,
\end{align}
where $\Psi_p:= \Psi_{\Pi_p(0)} = I - 2\frac{p \otimes p}{|p|^2}$.  In this paper we will only consider finite root systems $\P \subset \S^{N-1}$ that satisfy \eqref{as:ea}. 

\subsection{Interior and exterior cones} 
\label{sec:geo1}
We will show that sets satisfying \eqref{eqn:1as} are Lipschitz domains. To this end we define cones of directions. For $r > 0$ and a basis $A = \{ p_i \}_{i=1}^N \subset \S^{N-1}$ of $\Rn$ we write $\co_{r}(A)$ to denote the open {\it $r$-cone} generated by $A$:
\begin{align}
\label{eqn:par}
\co_r(\{p_i\}_{i=1}^N) = \left\{ \sum_{i=1}^N a_i p_i : \sum_{i=1}^N a_i < r \hbox{ and } a_i > 0 \hbox{ for all } 1 \leq i \leq N  \right\}.
\end{align}
That is, $\co_{r}(A)$ is the open $N$-simplex given by the convex hull $\operatorname{co} (rA \cup \{0\})$.

\medskip

For the rest of this paper, $\sigma_1$, $\sigma_2$ and $\sigma_3$ denote the following constants that characterize the distribution of the directions in $\P$:
\begin{align}
\label{eqn:s1}
\sigma_1 &= \sigma_1(\P) := \min_{p_1 \in \P} \max_{ \hbox{ a basis } \{p_i\}_{i=1}^N \subset \P} \min_{1 \leq i \leq N} |p_i \cdot p_1|,\\
\label{eqn:s2} 
\sigma_2 &= \sigma_2(\P) := \max_{x\in D} |x|,  \quad  D:=  \{x\in\Rn, \ | p \cdot x| \leq 1 \hbox{ for all } p \in \P \},\\
\label{eqn:s3}
\sigma_3 &= \sigma_3(\P) :=  \min_{\hbox{a basis } A \subset \P} \max \{r : B_r(\tfrac{1}{2N} \textstyle\sum_{p\in A}p) \subset \co_1(A)\}.
\end{align}

\begin{lem}
\label{lem:rv}
Let $\P \subset \S^{N-1}$ be a finite root system satisfying \eqref{as:ea}. Then $\sigma_1(\P) , \sigma_3(\mathcal{P}) \in (0, 1]$ and $\sigma_2(\P) \in [1, \infty)$.\end{lem}

\begin{proof}
It is clear that $\sigma_1 \leq 1$ as $\P \subset \S^{N-1}$. To show the lower bound, fix $p_1 \in \P$. By \eqref{as:ea}, we can find a basis $A \subset \P \setminus p_1^\perp$, and we can assume that $p_1 \in A$. But then $\min_{p \in A} |p \cdot p_1| > 0$. Since $\P$ is finite, we conclude that $\sigma_1 > 0$.
It is clear that $0<\sigma_3 <1$ since $\tfrac{1}{2N} \textstyle\sum_{p\in A}p \in \co_1(A)$ for any basis $A \subset \P$.

\medskip

To estimate $\sigma_2$ we set $K := \sup_{x\in D} |x|$. As $\P \subset \S^{N-1}$ we have $K \geq 1$. Let $\{x_i\}_{i \in \N}$ satisfy $|x_i| \to K$ with $|x_i| \geq 1$. As $\{x_i/|x_i|\} \subset \S^{N-1}$, along a subsequence  $x_i / |x_i| \to y \in \S^{N-1}$.
As $|p\cdot x_i| \leq 1$ for all $p \in \P$, we have
\begin{align*}
\left| p \cdot \frac{x_i}{|x_i|} \right| \leq \frac{1}{|x_i|}.
\end{align*}
Thus if $K = \infty$ it follows that $p \cdot y = 0$ for all $p \in \P$ with $y\in \S^{N-1}$, but this
contradicts \eqref{as:ea}. Hence we conclude that $1 \leq \sigma_2 = K < \infty$. 
\end{proof}

Now we are ready to state the main result in this section. 

\begin{thm}
\label{thm:lpb}
Suppose that $\Om$ satisfies \eqref{eqn:1as} and contains 
\begin{align}
\label{eqn:b}
\B_r := \sigma_1^{-1}\sigma_2 (\rho + 2r) B_1(0)
\end{align}
for some $r > 0$. Then  $\partial\Omega$ has the $r$-cone  property at every point, with locally constant cone directions that are independent of the choice of $\Om$.

More precisely, for every $x_0 \in \B_r^\cc$ there exists $A \subset \P$ that only depends on $x_0$ such that $A$ is a basis of $\R^N$ and
\begin{align}
\label{eqn:lpb12}
x + \co_{r}(A) \subset \Om^\mathsf{c} \hbox{ and } y - \co_{r}(A) \subset \Om
\end{align}
for any $x \in \Omega^\cc \cap B_r(x_0)$ and $y \in \Omega \cap B_r(x_0)$.
In particular, $\Om$ is a Lipschitz domain. 
\end{thm}  

When $\phi(x) = \psi(x) = |x|$, the above theorem corresponds to the cone property and star-shapedness of a set having \emph{$\rho$-reflection} (see \cite[Lemma 21]{FelKim14}).  

\newcommand{\squarefig}[1]{
\begin{tikzpicture}[scale=2.5]
\draw[thick]
(0:1)--(22.5:0.9)--(45:1)--
(45:1)--(67.5:0.9)--(90:1)--
(90:1)--(112.5:0.9)--(135:1)--
(135:1)--(157.5:0.9)--(180:1)--
(180:1)--(202.5:0.9)--(225:1)--
(225:1)--(247.5:0.9)--(270:1)--
(270:1)--(292.5:0.9)--(315:1)--
(315:1)--(337.5:0.9)--(360:1)--
cycle;
\def\r{0.2}
\def\rs{\r*(2*sqrt(2 - sqrt(2)))}
\draw[fill=white!80!blue] (0,0) circle [radius=\r];
\draw[dashed] (0,0) circle [radius=\rs];

\begin{scope}[rotate=#1]
\draw[thin,dashed] (\r,-1)--(\r,1) (-1,\r)--(1,\r);
\def\s{0.707}
\draw[fill=gray!50!white] (\r,\s) -- (45:1) -- (\s,\r);
\draw[fill=gray!50!white] (2*\s-\r,\s) -- (45:1) -- (\s,2*\s-\r);
\end{scope}

\draw (135:1) node[above left] {$\Omega$};
\draw (0,0) node {$B_{\rho}(0)$};
\draw (-45:{\rs}) node[right=5pt,fill=white] {$B_{2\sqrt{2 - \sqrt2}\rho}(0)$};
\end{tikzpicture}
}
\begin{figure}
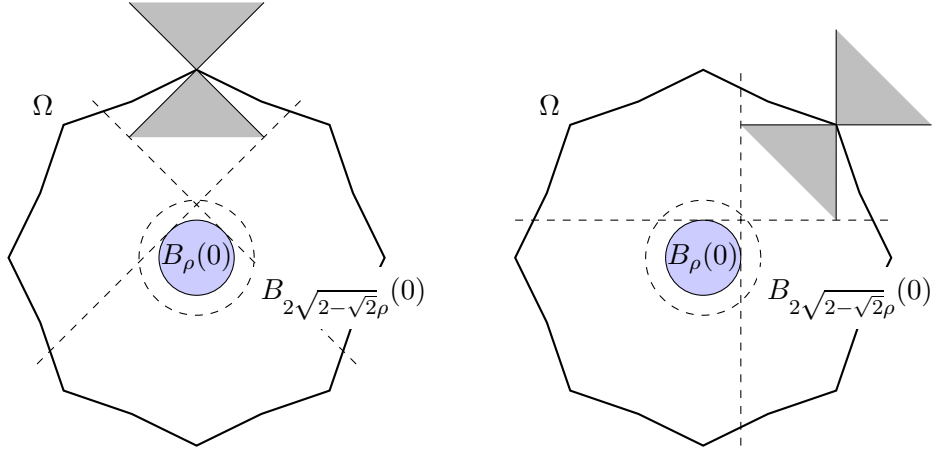

\centering
\begin{subfigure}[t]{0.4\textwidth}
\centering
\squarefig{45}
\end{subfigure}
\begin{subfigure}[t]{0.4\textwidth}
\centering
\squarefig{0}
\end{subfigure}
\caption{The cone property}
\label{fig:sq}
\end{figure}

\begin{ex}
\label{ex:sq}
In $\R^2$, recall $\P$ given in \eqref{eqn:sq}:
\begin{align*}
\P = \left\{ \pm e_1, \pm e_2, \tfrac{\pm 1}{\sqrt{2}}(e_1 + e_2), \tfrac{\pm 1}{\sqrt{2}}(e_1 - e_2) \right\} = \big\{ (\cos (k\pi / 4), \sin(k \pi/4)): 0 \leq k < 8\big\}.
\end{align*}
Then $\sigma_1(\P) = \cos(\pi/4) =  \frac{1}{\sqrt{2}}$ and $\sigma_2 (\P) = \cos(\pi/8)^{-1} = \sqrt{4 - 2\sqrt{2}}$ since $D$ is the regular octagon
$$
D = \{ (x_1, x_2) :  |x_1|, |x_2| \leq 1, |x_1 + x_2|, |x_1 - x_2| \leq \sqrt{2}\}.
$$ 
By Theorem~\ref{thm:lpb}, if $\Om$ satisfies \eqref{eqn:1as} and $\Om$ contains $B_{2\sqrt{2 - \sqrt{2}}\rho}(0)$, then $\Om$ has exterior and interior cones in $\Rn$ at any point $x \in \partial \Omega$ (see Figure~\ref{fig:sq}). 
\end{ex}

The proof of Theorem~\ref{thm:lpb} combines the following two geometric observations. For the rest of the section we assume that $\Om$ satisfies \eqref{eqn:1as}, or its equivalent form \eqref{eqn:rp}.

\begin{lem}
\label{lem:nv}
For $x \in \Rn$ with $|x| \geq \sigma_1^{-1} \sigma_2$, there exists a basis $A \subset \P$ such that
\begin{align}
\label{eqn:2nv}
p \cdot x  \geq  1 \quad \hbox{ for all } p \in A.
\end{align}
\end{lem}

\begin{proof}
Suppose $|x|\geq \sigma_1^{-1}\sigma_2$. Then $| p_1 \cdot x | \geq  \sigma_1^{-1}$ for some $p_1\in \P$, and by \eqref{eqn:sym} we can assume
\begin{align*}
p_1 \cdot x  \geq  \sigma_1^{-1}.
\end{align*}

Next observe that \eqref{eqn:sym} and the definition of $\sigma_1$ yields $\{p_i\}_{i=2}^N\subset \P$ that with $p_1$ span $\R^N$ such that
\begin{align*}
p_1 \cdot p_i \geq \sigma_1 \quad\hbox{ for all } 1 \leq i \leq N.
\end{align*}
Therefore
\begin{align*}
(2(p_1 \cdot p_i)p_1 - p_i) \cdot x + p_i \cdot x  = 2 (p_1 \cdot p_i) (p_1 \cdot x) \geq 2 \sigma_1 \sigma_1^{-1} = 2,
\end{align*}
which implies that at least one of the terms on the left is $\geq 1$.
%

Recall from \eqref{eqn:sym} that $2(p_1 \cdot p_i)p_1 - p_i \in \P$. For each $2 \leq i \leq N$, we replace $p_i$ in $\{p_i\}_{i=1}^N$ by $2(p_1 \cdot p_i)p_1 - p_i$ if  $p_i \cdot x < 1$. This new basis satisfies \eqref{eqn:2nv}.
\end{proof}

\begin{remark}
\label{rem:pri}
We point out that \eqref{as:ea} is essential for Lemma~\ref{lem:nv}, since if $\Pi$ is a hyperplane such that $\spn(\P \setminus \Pi) \neq \Rn$ then Lemma~\ref{lem:nv} does not hold for $x \perp \Pi$ no matter how large $|x|$ is. For instance if $N=2$ and $\P = \{\pm e_1, \pm e_2\}$, then $\P$ does not satisfy \eqref{as:ea} because
\begin{align*}
\P \setminus \Pi_{e_1}(0) = \{\pm e_1\}.
\end{align*}
\end{remark}

\begin{lem}
\label{prop:lip}
Suppose $x\in\Rn$ satisfies
\begin{align*}
\min_{p\in A} p \cdot x  \geq \rho+ r \qquad \text{for some basis } A \subset \P \hbox{ and } r>0.
\end{align*}
Then for $\co_r$ as given in \eqref{eqn:par}, we have
\begin{align*}
x + \co_{r}(A) \subset \Om^\mathsf{c} \hbox{ if } x\in \Om^\mathsf{c}\quad \hbox{ and } \quad x - \co_{r}(A) \subset \Om \hbox{ if } x\in \Om.
\end{align*}
\end{lem}

\begin{proof}
 We first prove that 
 \begin{equation}\label{26claim00}
 \hbox{ If }  y \cdot p > \rho \hbox{ and }  y + ap \in \Omega \hbox{ for some } a>0, \hbox{ then } y \in \Omega.
 \end{equation}
 This follows from \eqref{eqn:rp} with reflection $\Psi := \Psi_{\Pi_p(y \cdot p +\frac a2)}$. Since 
 $$
 \rho < y \cdot p < y\cdot p + \frac a2 < y \cdot p + a,
 $$ \eqref{eqn:rp} yields $y = \Psi(y + ap) \in \Omega$ if $y +ap \in \Omega$. 
  
\medskip
  
Let  us denote $A=\{p_i\}_{i=1}^N$ and choose $\{a_k\}_{k=1}^N$ as in \eqref{eqn:par}. We now apply \eqref{26claim00} iteratively to $y_k = x +\sum_{i=k+1}^N a_ip_i$, $a_k$ and $p_k$, for $1 \leq k \leq N$. This is possible since
$$
y_k \cdot p_k - \rho \geq r +\sum_{i=k+1}^N a_ip_i \cdot p_k > 0 \quad \hbox{ for }  1 \leq k \leq N.
$$
Since $x +\sum_{i=1}^N a_i p_i = y_1 + a_1 p_1$, $y_{k-1} = y_k + a_kp_k$, $y_N = x$, we deduce that $x + \sum_{i=1}^N a_i p_i \in \Omega$ implies $x \in \Omega$, or equivalently, $x \in \Omega^\mathsf{c}$ implies $x + \sum_{i=1}^N a_i p_i \in \Omega^\mathsf{c}$.

\medskip

 We next apply \eqref{26claim00} to $y_k = x - \sum_{i=1}^k a_i p_i$ , $a_k$ and $p_k$. Noting that 
 $$
 x = y_1 +a_1p_1, y_{k-1} = y_k + a_k p_k \hbox{ and }  y_N = x -  \sum_{i=1}^N a_i p_i,
 $$ we conclude that $x \in \Omega$ implies $x -  \sum_{i=1}^N a_i p_i \in \Omega$.
\end{proof}

\textbf{Proof of Theorem~\ref{thm:lpb}.}

\medskip

From Lemma~\ref{lem:nv}, for any point $x_0 \in \Rn \setminus \B_r$ there is a basis $A  \subset \P$ such that
\begin{align*}
p \cdot x_0  \geq  \rho + 2r \hbox{ for all } p \in A.
\end{align*}
Since $|p \cdot x - p \cdot x_0 | \leq |x-x_0|$  for all   $p \in A$, we have
\begin{align*}
p \cdot x \geq p \cdot x_0 - r \geq \rho + r \hbox{ for all } p \in A \hbox{ and } x\in B_r(x_0),
\end{align*}
We can now conclude by Lemma~\ref{prop:lip}. \hfill$\Box$

\subsection{In and out-radius}
\label{sec:geo2}
In this section, we estimate the in-radius of sets satisfying \eqref{eqn:1as} in terms of their out-radius. Recall that $\sigma_i$ are given by \eqref{eqn:s1}--\eqref{eqn:s3}.

\begin{thm}
\label{thm:bal}
Suppose $\Omega$ satisfies \eqref{eqn:1as} and
$$
\Om \nsubseteq B_R(0) \hbox{ for some } R> \sigma_1^{-1}\sigma_2(1 + \sigma_2\sigma_3^{-1})\rho.
$$
Then $B_c(0) \subset \Om$, where $c$ satisfies $R= R(c)= \sigma_1^{-1} \sigma_2 \big(\rho +  (\sigma_2 \rho  +  c) \sigma_3^{-1}\big)$.
\end{thm}

The proof proceeds in a number of steps. For the rest of the section we assume that $\Om$ satisfies \eqref{eqn:1as}, or its equivalent form \eqref{eqn:rp}.

\begin{prop}
\label{prop:ib}
Suppose there is  $x \in \Om$ such that  $|x| \geq \sigma_1^{-1} \sigma_2 (\rho + c\sigma_3^{-1})$ for some $c>0$. Then 
$$
B_c(z_0) \subset \Om \hbox{ for some } z_0 \in \Om\hbox{ such that }| x - z_0| \leq \frac{c }{2\sigma_3}.
$$
\end{prop}

\begin{proof}
From Lemma~\ref{lem:nv}, there is a basis $A \subset \P$ such that
$$
p \cdot x \geq \rho + c\sigma_3^{-1} \hbox{ for all } p \in A.
$$
Then Lemma~\ref{prop:lip} yields that
\begin{align*}
x - \co_{c\sigma_3^{-1}}(A) \subset \Om.
\end{align*}
It follows that $z_0:= x - \frac{c \sigma_3^{-1}}{2 N} \sum_{p \in A} p$ satisfies
\begin{align*}
B_{c} (z_0) \subset x - \co_{c\sigma_3^{-1}}(A) \subset \Om.
\end{align*}
\end{proof}

Next we show that the reflection of an interior ball of $\Om$ is also in $\Om$.

\begin{lem}
\label{lem:br}
For $p \in \P$, $s>\rho$ and $z \in \Pi_p^+(s)$,
\begin{align*}
\hbox{ if } B_c(z) \subset \Om, \hbox{ then } B_c(\Psi_{\Pi_p(s)}(z)) \subset \Om.
\end{align*}
\end{lem}
\begin{figure}
\centering
\begin{tikzpicture}[scale=0.2]
\fill[blue!2!white] plot [smooth] coordinates { (-8,4) (-3,5) (2,6) (12,3) } -- (12, -8) -- (-8, -8) -- cycle;
\draw (10,3) node[above right] {$\partial \Omega$};
\draw (1,2) circle [radius=3];
\draw (4,2) node[right] {$B_c(z)$};
\fill (1,2) node[right] {$z$} circle[radius=5pt];
\draw (-1,-2) circle [radius=3];
\draw (-1,-5) node[below right] {$\Psi_{\Pi_{p}(s)}(B_c(z))$};
\draw (-6,3) -- (6,-3) node[right] {$\Pi_{p}(s)$};
\draw[thick] plot [smooth] coordinates { (-8,4) (-3,5) (2,6) (12,3) };
\end{tikzpicture}
\caption{Ball reflection in the proof of Lemma~\ref{lem:br}.}
\label{fig:ball-reflection}
\end{figure}
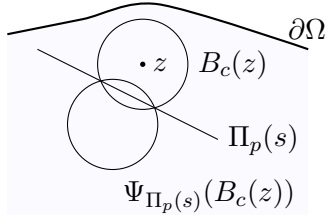
\begin{proof}
The argument is illustrated in Figure~\ref{fig:ball-reflection}. We have
\begin{align*}
\Psi_{\Pi_{p}(s)} ( B_c(z) \cap \Pi^+_{p}(s) ) =   B_c(\Psi_{\Pi_{p(s)}}(z)) \cap \Pi^-_{p}(s).
\end{align*}
As $B_c(z) \subset \Om$  and $\Om$ satisfies \eqref{eqn:rp}, we have
\begin{align*}
B_c(\Psi_{\Pi_{p(s)}}(z)) \cap \Pi^-_{p}(s)
\subset \Psi_{\Pi_{p}(s)} (\Om \cap \Pi^+_{p}(s)) \subset \Om \cap \Pi^-_{p}(s).
\end{align*}

\medskip

It remains to show that
\begin{align}
\label{eqn:ibt26}
B_c(\Psi_{\Pi_p(s)}(z)) \cap \left( \Pi^+_{p}(s) \cup \Pi_{p}(s) \right) \subset \Om.
\end{align}
Since $z \in \Pi_p^+(s)$ and  $\Psi_{\Pi_p(s)}(z) \in \Pi_p^-(s),$
we have
\begin{align*}
| y - z | \leq | y - \Psi_{\Pi_p(s)}(z) | \hbox{ for any } y \in \Pi^+_{p}(s) \cup \Pi_{p}(s)
\end{align*}
and thus we conclude \eqref{eqn:ibt26} by
\begin{align*}
B_c(\Psi_{\Pi_p(s)}(z)) \cap \left( \Pi^+_{p}(s) \cup \Pi_{p}(s) \right) \subset B_c(z) \subset \Om.
\end{align*}
\end{proof}

\begin{prop}
\label{prop:ibt}
Suppose $ B_c(z_0) \subset \Om$ for some  $z_0 \in \Om$ and $c> \sigma_2\rho$.
Then
\begin{align}
\label{eqn:2ibt}
B_{c-\sigma_2 \rho}(0) \subset \Om.
\end{align}
\end{prop}

\begin{proof}
For a given $\e>0$, let $s:= (1+\e)\rho$ and construct sequences $\{z_m\}_{m \geq 0} \subset \Rn$ and $\{q_m\}_{m \geq 0} \subset \P$ by
\begin{align*}
z_{m+1} := \Psi_{\Pi_{q_{m}(s)}}(z_{m}) \quad  \hbox{and}\quad q_m \in \argmax_{ p \in \P }  p \cdot z_m .
\end{align*}
A direct  computation yields
$$
| z_{m+1} |^2 = | z_m - 2(z_m \cdot q_m) q_m + 2s q_m |^2 = | z_m |^2 + 4 s^2 - 4 s (z_m \cdot q_m).
$$
If we have $q_m \cdot z_{m} > s + \e$, then from the above equation we have
$$
| z_{m+1} |^2 - |z_m|^2 < -4s\e.
$$
Since $|z_m| \geq 0$ for all $m \geq 0$, there exists the smallest $m^* \geq 0$ such that $q_{m^*} \cdot z_{m^*} \leq s+\e$. Since $q_m$ is a maximizer of $p \cdot z_m$ in $\P$, we conclude that 
\begin{align}
\label{eqn:ibt12}
| z_{m^*} |  \leq \sigma_2 (s + \e).
\end{align}

Recall that $B_c(z_0) \subset \Omega$. Also we have $q_m \cdot z_m > s+\e$ for $0 \leq m<m^*$ by our choice of $m^*$, which implies $z_m \in \Pi_{q_m}^+(s)$  for $0 \leq m < m^*$.  Thus applying Lemma~\ref{lem:br} iteratively, we arrive at 
$$
B_c(z_{m+1}) = B_c(\Psi_{\Pi_{q_m}(s)}(z_m)) \subset \Omega \quad \hbox{ for } 0\leq m  \leq m^*-1,
$$ 
This and  \eqref{eqn:ibt12} yield
\begin{align*}
B_{c - \sigma_2(s+\e)}(0) \subset \Om.
\end{align*}
As $\e$ is arbitary, we conclude.
\end{proof}

\textbf{Proof of Theorem~\ref{thm:bal}.}

\medskip

From Proposition~\ref{prop:ib}, there exists $z_0 \in \Om$ such that
\begin{align*}
B_{\sigma_2 \rho + c}(z_0) \subset \Om.
\end{align*}
Proposition~\ref{prop:ibt} now yields $B_{\sigma_2 \rho + c - \sigma_2 \rho }(0) \subset \Om$.
\hfill$\Box$

\section{Viscosity solutions}
\label{sec:not}

In this section we introduce a notion of viscosity solutions for the level set equations of \eqref{model} when $\phi$ is smooth, namely when $\phi$ is $C^2$ away from the origin. Readers may skip this section during the first reading, as it does not directly relate to the geometric description of the results. 

Let us first recall some standard notations.

\begin{itemize}
\item $Q:=\Rn \times (0,\infty) \hbox{ and } Q_T:=\Rn \times (0,T], \qquad T > 0.$
\item For a domain $U \subset \Rn$ and $T>0$,  $U_T := U \times (0,T],$
$\partial_p U_T := (U \times \{0\}) \cup (\partial U \times [0, T])$.
\item For $(x_0,t_0) \in \R^N \times (0,\infty)$,
$$D_r(x_0,t_0) := B_r(x_0) \times (t_0 - r^2,t_0], \quad \partial_p D_r := (\oB_r(x_0) \times \{ t_0 - r^2 \}) \cup ( \partial B_r(x_0) \times [t_0 - r^2, t_0]).
$$
\item For $u : L \subset \R^{N + 1} \rightarrow \mathbb{R}$, we denote its semi-continuous envelopes $u_* , u^*: \oL \rightarrow \R$ by 
\begin{align}
\label{eqn:env}
u_*(x,t) := \lim_{\e \to 0+} \inf_{B_\e(x,t) \cap L} u\quad \hbox{ and } u^*(x,t) := \lim_{\e \to 0+} \sup_{B_\e(x,t) \cap L} u.
\end{align}

\end{itemize}

When $\lambda$ is continuous, the level set equation of \eqref{model}  can be written as 
\begin{align}
\label{eqn:the}
u_t =  \psi(-Du) (-\dv D\phi (-Du) + \lambda),
\end{align}
assuming $\Omega_t = \{u(\cdot,t)>0\}$. 
In this case we can use the following definition used in \cite{KimKwo18b} that proves convenient for stability properties. It can be shown to be equivalent to the more classical version in \cite{CCG91} using the argument in \cite[Sec.~7.2]{CafSal05}. 

\medskip

Recall that $\varphi\in C^{2,1}(D_r)$ is said to be a \emph{classical strict subsolution} (resp. \emph{supersolution}) of \eqref{eqn:the} if $\varphi_t < F_*(t, D\varphi, D^2\varphi)$ (resp. $\varphi_t > F^*(t, D\varphi, D^2\varphi)$) on $D_r$ with the right-hand side of \eqref{eqn:the} written as $F(t, p, X) := \psi(-p)(\trace (D^2_p\phi(-p) X) + \lambda(t))$, $p \neq 0$.

\begin{DEF}
\label{def:visd}
Let $\phi \in C^2(\Rn \setminus \{0\})$ and $\lam \in C([0,\infty))$.
\begin{itemize}
\item[$\circ$]
A function $u: Q\to \mathbb{R}$ is \textit{a viscosity subsolution} of \eqref{eqn:the} if  $u^*<\infty$ and for $D_r \subset Q$ and for every classical strict supersolution $\varphi \in C^{2,1}(D_r)$, $u^* < \varphi$ on $\partial_p D_r$ implies $u^* < \varphi$ in $\overline{D_r}$.
\item[$\circ$] A function $u: Q\to \mathbb{R}$ is \textit{a viscosity supersolution} of \eqref{eqn:the} if $u_*>-\infty$ and $D_r \subset Q$ and for every classical strict subsolution $\varphi \in C^{2,1}(D_r)$, $u_* > \varphi$ on $\partial_p D_r$ implies $u_* > \varphi$ in $\overline{D_r}$.
\end{itemize}
\end{DEF}

Since the forcing term $\lam$ in \eqref{model} may not exist in a classical sense, we cannot directly use the above notion of viscosity solutions. Indeed we will modify the notion to incorporate $\lambda$ as the distributional derivative of a continuous function $\Lambda$. To this end we use the set convolutions as follows. While the definition is mostly parallel to the isotropic case in \cite{KimKwo18b}, the geometric nature of the convolution is different.
 
 \medskip
 
 For $\gam \in C(\Rpz;\Rpz)$, the sup-convolution $\hu(\cdot;\gam)$ and inf-convolution $\tu(\cdot;\gam)$ over the Wulff shape $W_\psi$ are given by
\begin{align}
\label{eqn:sup}
\hu(x,t; \gam) := \sup_{x-\gam(t)W_\psi} u(\cdot,t)
\qquad\hbox{ and }\qquad
\tu(x,t; \gam) := \inf_{x+\gam(t)W_\psi} u(\cdot,t).
\end{align}
We often write $\hu(x,t)$ and $\tu(x, t)$ if $\gamma$ is understood from the context. Note that the sign $x - \gamma(t) W_\psi$ in \eqref{eqn:sup} is chosen so that 
$$\{\hu(\cdot, t; \gamma) > 0\} = \{u(\cdot, t; \gamma) > 0\} + \gamma(t) W_\psi.
$$
 As we do not assume that $\psi$ is even, in general $-W_\psi \neq W_\psi$.  See Figure~\ref{fig:supinfconv} for illustration.
 Note also that $\widehat{u^*} = (\widehat{u})^*$ and $\widetilde{u_*}  = (\widetilde{u})_*$ (see \cite[Lemma C.8]{KimKwo18b}).  
 
In contrast to the isotropic convolution over the balls used in \cite{KimKwo18b},  the sup-convolution and inf-convolution over the Wulff shape of $\psi$ are needed to modify the normal velocity of the level sets while accounting for the anisotropic factor $\psi$ (see Lemma~\ref{le:wulff-sup-conv}). This is the main feature that differs from the notion introduced in \cite{KimKwo18b}.

\medskip

Using these convolutions we now define viscosity solutions of the level set equation
\begin{align}
\label{main}
u_t  &= \psi(-Du) (-\dv D\phi (-Du) + \Lambda') \,\,\, \hbox{ in } Q
\end{align}
for the flow
\begin{align}
\label{main-flow}
V=\psi(\n)(-\k_\phi+\Lam').
\end{align}

\begin{DEF}
\label{def:visl} For a function $u:Q\to\mathbb{R}$,
\begin{itemize}
\item[$\circ$]
$u$ is \textit{a viscosity subsolution} of \eqref{main} if  $u^*<\infty$ and for any $0 \leq t_1 < t_2$ and $ \Theta \in C([t_1, t_2]) \cap C^1((t_1,t_2))$ such that $\Theta > \Lambda$ in $[t_1,t_2]$, the function $\hu = \hu(\cdot; \Theta - \Lam)$ given in \eqref{eqn:sup} is a viscosity subsolution of
\begin{equation}
\label{eqn:the2}
u_t =  \psi(-Du) (-\dv D\phi (-Du) + \Theta') \,\,\,\hbox{ in } (t_1, t_2) \times \Rn.
\end{equation}
\item[$\circ$]
$u$ is \textit{a viscosity supersolution} of \eqref{main} if  $u_*> -\infty$ and for any $0 \leq t_1 < t_2$ and $\Theta \in C([t_1, t_2]) \cap C^1((t_1,t_2))$ such that $\Theta < \Lambda$ in $[t_1,t_2]$, the function $\tu = \tu(\cdot; -\Theta+\Lam)$ given in \eqref{eqn:sup} is a viscosity supersolution of \eqref{eqn:the2}.\\ \\
\item[$\circ$] 
 $u$ is \textit{a viscosity solution} of \eqref{main} with initial data $u_0$ if $u^*$ and $u_*$ are respectively \textit{a viscosity subsolution} and \textit{a viscosity supersolution} of \eqref{main}, and if $u^*= {(u_0)}^*$ and $u_*= {(u_0)}_*$ at $t=0$. \\ \\
\item[$\circ$]
$\ott$ is \textit{a viscosity solution (subsolution or supersolution, respectively)} of \eqref{main-flow} if there exists \textit{a viscosity solution (subsolution or supersolution, respectively)} $u$ of \eqref{main} such that
\begin{align*}
\ot = \{ u(\cdot,t) > 0 \} \,\,\,\hbox{ for all } t \geq 0.
\end{align*}
\end{itemize}
\end{DEF}
To simplify the notation, we will sometimes say that $u: Q \to \R$ is a viscosity solution of the flow \eqref{main-flow} if it is a viscosity solution of the associated level set equation \eqref{main}.

\begin{figure}
\centering
\begin{tikzpicture}
\draw (0, 1) -- (-0.866, -0.5) -- (0.866, -0.5) -- cycle;
\draw (-1,1.05) -- (1,0.95);
\draw (-1,0.05) -- (1,-0.05);
\draw[dotted] (0,0)--(0.05,1);
\draw[<-] (0.025, 0.5) node[below right] {$\n$} -- (0, 0);
\draw[<->] (-0.75, 0.0375) -- node[left] {$\psi(\n)$} (-0.75 + 0.05, 0.0375 +1);
\fill (0, 0) node[below] {$0$} circle [radius=1pt];
\draw (0, -0.5) node[below] {$W_\psi$};
\begin{scope}[xshift=6cm]
\draw (-2, 0.1) -- (2, -0.1);
\fill[black!10!white] (-2, 0.1) -- (2, -0.1) -- (2, -1) -- (-2, -1) -- cycle;
\draw (-2, 0.7) -- (2, 0.5);
\begin{scope}[xshift=1.5cm,yshift=0.525cm,scale=0.25]
\draw[<-] (0.1, 2) node[above] {$\n$} -- (0, 0);
\end{scope}
\begin{scope}[xshift=-1.5cm,yshift=0.125cm]
\draw[<->] (0, 0) -- node[left] {$\gamma \psi(\n)$} (0.1 * 0.25, 2 * 0.25);
\end{scope}
\draw (0.8, 0.55) node [below right] {$\hu > 0$};
\draw (0.8,-0.05) node[below right] {$u > 0$};
\begin{scope}[scale=-0.6,yshift=-1cm]
\draw (0, 1) -- (-0.866, -0.5) -- node[above] {$x-\gamma W_\psi$} (0.866, -0.5) -- cycle;
\draw (0, 0) node[below] {$x$} circle [radius=1pt/0.6];
\end{scope}
\end{scope}
\end{tikzpicture}
\caption{Illustration why the sup-convolution over $W_\psi$, $\hu(\cdot, \cdot; \gamma)$, moves the boundary of super-level sets by distance $\gamma\psi(\n)$ in the outer normal $\n$ direction.}
\label{fig:supinfconv}
\end{figure}
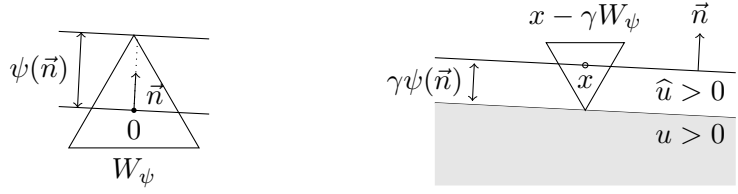

\medskip

By Definition~\ref{def:visl}, the following property holds for the set solutions. 
\begin{lem}
\label{lem:conv}
Let $\Lam \in C([0,\infty))$ and $\gam \in C(\Rpz;\Rpz)$, and denote
\begin{align*}
\hOm_t(\gam) := \bigcup_{y \in \gam(t)W_\psi} ( \Om_t + y ) \hbox{ and } \tOm_t(\gam) := \bigcap_{y \in \gam(t)W_\psi} ( \Om_t - y ).
\end{align*}
If $\ott$ is a viscosity subsolution of $V=\psi(\n)(-\k_\phi+\Lam')$, then $\hOm_t(\gam)$ is a viscosity subsolution of $V=\psi(\n)(-\k_\phi+\Lam'+\gam')$.

Similarly, if $\ott$ is a viscosity supersolution of $V=\psi(\n)(-\k_\phi+\Lam')$, then $\tOm_t(\gam)$ is a viscosity supersolution of $V=\psi(\n)(-\k_\phi+\Lam'-\gam')$.
\end{lem}

\begin{proof}
We only show the subsolution part, since the rest can be shown with parallel arguments. For a given $\Theta \in C([t_1, t_2]) \cap C^1((t_1,t_2))$ such that $\Theta > \Lam + \gam$ in $[t_1,t_2]$,  we need to show that 
\begin{align}
\label{eqn:cov11}
\hbox{ $\widehat{\hOm_t(\gam)}(\Theta - \Lam - \gam)$ is a viscosity subsolution of $V=\psi(\n)(-\k_\phi+\Theta')$ in $(t_1, t_2)$. } 
\end{align}
If this is the case, we conclude that $\hOm_t(\gam)$ is a viscosity subsolution of $V=\psi(\n)(-\k_\phi+\Lam'+\gam')$.
\medskip

As $W_\psi$ is convex, we get
\begin{align*}
\widehat{\hOm_t(\gam)}(\Theta - \Lam - \gam) = \bigcup_{\substack{y \in \gam(t)W_\psi, \\ z \in (\Theta - \Lam - \gam)(t)W_\psi}} ( \Om_t + y + z ) = \bigcup_{w \in (\Theta - \Lam)(t)W_\psi} (\Om_t + w ) = \hOm_t(\Theta - \Lam).
\end{align*}
As $\Theta > \Lam$ in $[t_1, t_2]$ and $\ott$ is a viscosity subsolution of $V=\psi(\n)(-\k_\phi+\Lam')$, it follows that $\hOm_t(\Theta - \Lam)$ is a viscosity solution of $V=\psi(\n)(-\k_\phi+\Theta')$ in $(t_1, t_2)$. We can now conclude.
\end{proof}

\begin{DEF}
\label{def:model}
For $\Lam\in C([0,\infty))$, the pair $(\ott,\Lam)$ is a viscosity solution of \eqref{model} if
\begin{align*}
\ott \hbox{ is a viscosity solution of } V=\psi(\n)(-\k_\phi+\Lam') \hbox{ and } |\Om_t| = |\Om_0|.
\end{align*}
\end{DEF}

Below we present the comparison principle, stability and well-posedness properties for our notion of viscosity solutions. The proofs are omitted as they are parallel to those of \cite[Theorems~2.10, 2.12 \& 2.14]{KimKwo18b}.

\begin{thm}
\label{thm:comp}
Given $\phi \in C^2(\R^N \setminus \{0\})$ and $\Lam \in C([0,\infty))$, let $u$ and $v$ be each a viscosity subsolution and a viscosity supersolution of \eqref{main}. Then the following holds for any bounded domain $U$ in $\Rn$ and $T>0$: If $u^* \leq v_* \hbox{ on } \partial_p U_T$, then $u^* \leq v_*$ in $U_T$.
\end{thm}

For a sequence of functions $\{ u_k \}_{k \in \N}$ on $Q$, the half-relaxed limits are defined as
\begin{align*}
\limsups_{k \rightarrow \infty}u_k(x,t) &:= \lim\limits_{j \to \infty} \sup \left\{ u_k(y,s) : k \geq j,\quad |y-x| \leq \tfrac{1}{j}, \quad |s-t| \leq \tfrac{1}{j} \right\},\\
\liminfs_{k \rightarrow \infty}u_k(x,t) &:= \lim\limits_{j \to \infty} \inf \left\{ u_k(y,s) : k \geq j,\quad |y-x| \leq \tfrac{1}{j}, \quad |s-t| \leq \tfrac{1}{j} \right\}.\nonumber
\end{align*}

\begin{thm}
\label{thm:sta}
Let $\phi\in C^2(\R^N \setminus \{0\})$, and let  $\{\Lambda_k\}_{k\in\N}$ be a sequence in  $C([0,\infty))$ that locally uniformly converges to $\Lambda_\infty$ as $k\to\infty$. Let $ u_k$ be a viscosity subsolution (supersolution, respectively) of \eqref{main} with $\Lambda = \Lambda_k$. If 
$$u:= \limsups_{k \rightarrow \infty} u_k<\infty \,\,\,( u:=\liminfs_{k \rightarrow \infty}u_k > -\infty, \hbox{respectively}),
$$ then $u$ is a viscosity subsolution (supersolution, respectively) of \eqref{main} with $\Lam = \Lambda_\infty$.
\end{thm}
As in \cite{CCG91} we denote
$$
C_a(\R^N) = \{u: u-a\in C_c(\R^N)\}.
$$
\begin{thm}
\label{thm:vex}
Let $\Lam \in C([0,\infty))$, $\phi \in C^2(\R^N \setminus \{0\})$ and $g \in C_a(\Rn)$ for some $a \in \R$. Then for any $T > 0$ there is a unique viscosity solution $u \in C_a(\overline{Q_T})$ of \eqref{main} with initial data $g$.
\end{thm}

\section{Preservation of the reflection property}
\label{sec:pre}

Here we show that the curvature flow with fixed forcing \eqref{main} preserves \eqref{eqn:1as} for smooth $\phi$, when $\phi$ and $\psi$ are invariant with respect to reflections given by elements of $\P$, namely when they satisfy
\begin{align}
\label{eqn:rt}
\phi = \phi \circ \Psi_p \hbox{ and } \psi = \psi \circ \Psi_p \quad \hbox{ for any } p \in \P.
\end{align}

\medskip

Throughout this section we fix $a < 0$. For a given bounded domain $\oz$, consider $u_0 \in C_a(\Rn)$  given by
\begin{align}
\label{initial_a}
u_0(x) := \max \{ -\sd (x, \oz) , a \} \hbox{ for all } x \in \Rn.
\end{align}
Here $\sd$ denotes the signed distance function, namely
\begin{align*}
\sd (x, \Om) := 
\begin{cases}
- \dist(x,\partial \Om) &\hbox{ if } x \in \Om,\\
\dist(x,\partial \Om) &\hbox{ if } x \in \Om^\mathsf{c}.
\end{cases}
\end{align*}

\begin{prop}
\label{prop:ref}
Assume that $\phi \in C^2(\R^N \setminus \{0\})$. Let $u\in C_a(\overline{Q_T})$ be a viscosity solution of \eqref{main} with initial data $u_0$. If $\oz$ satisfies \eqref{eqn:1as}, then so does $\ot := \{ u(\cdot,t) > 0 \}$  for all $t \in [0,T]$.
\end{prop}

\begin{proof}
 For any fixed $p \in \P$ and $s \in \R$, we claim that
\begin{align*}
\hbox{ $v(x,t):= u(\Psi_{\Pi_p(s)}(x),t)$ is a viscosity solution of \eqref{main} }
\end{align*}
with the initial data $v_0(x) := u_0(\Psi_{\Pi_p(s)}(x))$. To verify the viscosity solution condition, we only need to show that sub- and supersolution property is preserved for reflections of test functions. Therefore it is enough to verify the invariance of \eqref{main} under reflection for functions $u\in C^2(Q_T)$ and $\Lam \in C^1(\Rpz)$, at points $(x,t) \in Q_T$ where $Dv(x,t) \neq 0$. First note that, identifying the linear operator as its gradient, $D\Psi_{\Pi_p(s)}(x) = I - 2p\otimes p =: \Psi_{\Pi_p(0)} =: R$, which is a unitary symmetric matrix, and $R^2 = I$. Therefore
\begin{align*}
Dv(x,t) = R Du(\Psi_{\Pi_p(s)}(x), t)
\end{align*}
and both sides have the same norm. Moreover, $\psi(-Dv(x,t)) = \psi(-Du(\Psi_{\Pi_p(s)}(x), t))$ since $\psi = \psi \circ R$ by assumption.

Finally, differentiating $\phi(x) = \phi(Rx)$ twice we obtain $R D^2\phi(x)R = D^2\phi(Rx)$. A similar computation yields $D^2 v(x, t) = R D^2 u(\Psi_{\Pi_p(s)}(x), t) R$. This implies
\begin{align*}
\dv (D\phi(-Dv(x,t))) &= \trace (-D^2\phi(-Dv(x,t)) D^2 v(x,t)) \\&= \trace (-RD^2\phi(-Du(\Psi_{\Pi_p(s)}(x), t)) R R D^2u(\Psi_{\Pi_p(s)}(x), t) R) \\&= \trace (-D^2\phi(-Du(\Psi_{\Pi_p(s)}(x), t)) D^2u(\Psi_{\Pi_p(s)}(x), t)).
\end{align*}
We conclude that all terms in \eqref{main} are invariant with respect to the reflection $\Psi_{\Pi_p(s)}$ for test functions, and hence $v$ is a viscosity solution whenever $u$ is.

\medskip

Let us now fix $p \in \P$ and  $s  >\rho$.
As $\oz$ satisfies \eqref{eqn:1as}, $u_0$ given in \eqref{initial_a} satisfies 
\begin{align*}
v_0(x) := u_0(\Psi_{\Pi_p(s)}(x)) \leq u_0(x) \qquad \hbox{for all } x \in \Pi_p^-(s).
\end{align*}
By definition $v = u$ on $\Pi_p(s) \times [0, T]$.
As $u \in C_a(Q_T)$, there exists a bounded domain $U$ in $\Rn$ such that $u = v = a$ in $U^\mathsf{c} \times [0,T]$. Applying the comparison principle Theorem~\ref{thm:comp} in $(U \cap \Pi_p^-(s)) \times [0,T]$, we conclude that
\begin{align*}
v \leq u \qquad \hbox{in } \Pi_p^-(s) \times [0,T],
\end{align*}
which implies \eqref{eqn:1as} for $\ot$.
\end{proof}

The following theorem will be used in the next section to guarantee \eqref{eqn:1as} for the discrete-time flow that approximates \eqref{model}. Let us define $\K>0$ by the $\sigma_i$'s given in Section~\ref{sec:geo}:
\begin{align}
\label{eqn:k0}
\K = \K(\P) := \sigma_1^{-1} \sigma_2 \big( 1+ \sigma_2 \sigma_3^{-1}(1 + \sigma_1^{-1})\big).
\end{align}

\begin{thm}
\label{thm:pre}
For a given $\phi \in C^2(\R^N \setminus \{0\})$, let $u \in C_a(\overline{Q_T})$ be a viscosity solution of \eqref{main} with initial data \eqref{initial_a}, where $\oz$ satisfies \eqref{eqn:1as}. Suppose that
\begin{align}
\label{eqn:1pre}
|\ot| > |B_{\K \rho}| + \e \quad \hbox{ for some } \e>0 \hbox{ and } t\geq0 \hbox{ where } \ot:= \{ x \in \R^N : u(x,t) > 0\}.
\end{align}
Then, $\ot$ satisfies \eqref{eqn:1as} and contains $\B_{r}$ given in \eqref{eqn:b} for some $r = r(\P, \e)>0$.
\end{thm} 

\begin{proof}
Due to Proposition~\ref{prop:ref} we only need to check that $\ot$ contains $\B_r$. For $R(\cdot)$ given in Theorem~\ref{thm:bal}, observe that from the definition of $\K$
\begin{align*}
R(\sigma^{-1} \sigma_2 \rho) = \sigma_1^{-1} \sigma_2 \big( \rho +  (\sigma_2 \rho  +  \sigma^{-1} \sigma_2 \rho) \sigma_3^{-1} \big) = \K \rho.
\end{align*}
From this and \eqref{eqn:1pre} there exists $\e_1 = \e_1 (\P, \e)>0$ such that
\begin{align*}
\ot \nsubseteq B_{R(\sigma_1^{-1} \sigma_2 \rho) + \e_1}(0).
\end{align*}
Since $\ot$ satisfies \eqref{eqn:1as},  Theorem~\ref{thm:bal}  yields $r = r(\P, \e)>0$ such that
\begin{align}
\label{eqn:pre11}
\B_r = B_{\sigma^{-1} \sigma_2 (\rho + 2r)}(0)  \subset \ot.
\end{align}
\end{proof}

\section{The discrete \texorpdfstring{$\lambda$}{lambda} scheme}
\label{sec:lam}
In this section we introduce an explicit way to construct the flow for \eqref{model} by approximation, where the approximate $\lambda$ is given as a piecewise constant function of time. As in Section~\ref{sec:pre}, we rely on the level set approach to describe the approximate flow. 

\medskip

 For a given bounded open $E \subset \Rn$, $a<0$ and $h>0$, let $u(\cdot; E, h) \in C_a(\overline{Q_h})$ be the unique viscosity solution of $V = \psi(\n)(-\k_\phi + \lambda(E, h))$ with initial data $u_0 \in C_a(\Rn)$ given by
\begin{align}
\label{freezing-init-data}
u_0(x) := \max \{ -\sd (x, E) , a \},
\end{align}
where the  forcing term is defined with a constant $\alpha>0$:
\begin{align}
\label{eqn:lam}
\lambda(E, h) := \frac{\alpha \sign(|\oz| - |E|)}{\sqrt{h}}.
\end{align}
Here $\sign(c)$ denotes the sign of $c$, set  to zero if $c$ equals zero. Note that the well-posedness of $u$ follows from Theorem~\ref{thm:vex}.  We then define the corresponding set evolution
\begin{align}
\label{eqn:tt}
\T_t(E;h) := \{ x \in \Rn : u(x,t; E,h) > 0 \}, \qquad t \in [0, h].
\end{align}

\medskip


Iterating this process over the time intervals $[kh, (k+1)h]$ for $k\in \N$, we define the evolution 
\begin{align}\label{eqn:et}
E_t(h) := \T_{t - h \lfloor\frac{t}{h}\rfloor} \left( (\T_h)^{\lfloor\frac{t}{h}\rfloor}(\oz;h); h \right), \,\,\lh(t) := \lambda(E_{\lfloor t/h \rfloor h}(h), h) \,\,\hbox{ and } \Lh(t):= \int_0^t \lh(s)ds.
\end{align}
 Here $\T_h^{m}$ denotes the $m$-th functional power of $\T_h(\cdot; h)$. 

\medskip

In the remainder of the paper we will study the properties of the $h$-flow $E_t(h)$ and its limit as $h$ tends to zero. The goal is to show that the limit  is a solution of the volume-preserving flow \eqref{model}. First we show that the $h$-flow is a viscosity solution with the corresponding forcing.

\begin{lem}
\label{lem:glue}
Let $E_t(h)$ and $\lh$ be given in \eqref{eqn:et}. For $h>0$, $E_t(h)$ is a viscosity solution of $V = \psi(\n)(-\kp + \lh)$.
\end{lem}

\begin{proof}
By construction, $E_t(h)$ is a viscosity solution of $V = \psi(\n)(-\kp + \lh)$ in $\Rn\times ((k-1)h, kh]$ for all $k \in \N$. It is thus enough to show that $E_t(h)$ is a viscosity solution of $V = \psi(\n)(-\kp + \lh)$ in $\Rn\times (t_1,t_2]$ for times $t_1 = h-\e$ and $t_2=h+\e$ with $0<\e <h$. 
 
 \medskip
 
 Following Definition~\ref{def:visl}, we check the viscosity subsolution property for $u(x,t) := \chi_{E_t(h)}(x)$. For any $ \Theta \in C([t_1, t_2]) \cap C^1((t_1,t_2))$ such that $\Theta > \Lambda$ in $[t_1,t_2]$, we need to verify that 
\begin{align*}
\hbox{ $\hu = \hu(\cdot ; \Theta - \Lam)$ is a viscosity subsolution of $V = \psi(\n)(-\kp + \Theta')$ in $\Rn\times (t_1,t_2]$, }
\end{align*}
where $\hu$ is given in \eqref{eqn:sup}. 
Consider a cylindrical domain $D_r \subset \Rn \times (t_1,t_2)$ and a classical strict supersolution $\varphi$ of $V = \psi(\n)(-\kp + \Theta')$ that is above $\hu^*$ on the parabolic boundary of $D_r$. In this setting we would like to show that  $\hu^* < \varphi$ in $\overline{D_r}$.

\medskip

By Lemma~\ref{lem:conv}, $\hu(\cdot ; \Theta - \Lam)$ is a viscosity subsolution $V = \psi(\n)(-\kp + \Theta')$ in $\Rn\times (t_1,h]$ and $\Rn\times (h,t_2]$.  Since $\varphi$ is above $\hu^*$ on the parabolic boundary of $D_r \cap \{ t\leq h\}$, it follows that 
$$
\hu^* < \varphi \hbox{ in } \overline{D_r} \cap \{t\leq h\}.
$$
From the above inequality, we now have $\varphi$ above $\hu^*$ on the parabolic boundary of $D_r \cap \{t\geq h\}$. Thus we conclude that 
$$
\hu^* < \varphi \hbox{ in } \overline{D_r}\cap \{t\geq h\}.
$$
Thus we have shown the subsolution property for $u$.

\medskip

Parallel arguments yield that $u$ is a viscosity supersolution of $V = \psi(\n)(-\kp + \lh)$.
\end{proof}

\medskip

 Next we show that the $h$-flow keeps its volume close to that of the initial set $|\Omega_0|$ if the constant $\alpha$ in the forcing term \eqref{eqn:lam} is sufficiently large (Theorem~\ref{thm:ve}). More precisely we require that
\begin{equation}\label{eqn:al}
\alpha = \max\left\{\frac{M_\phi}{\sigma_3 m_{\psi}}, 1 \right\} \left(\frac{N-1}{m_{\phi/\psi}}\right)^{\frac{1}{2}},
\end{equation}
where 
\begin{align}
\label{eqn:mphi}
M_{f} := \sup_{\nu \in \S^{N-1}} f(\nu) \qquad \hbox{and}\qquad m_{f}:= \inf_{\nu \in \S^{N-1}} f(\nu).
\end{align}

\begin{thm}
\label{thm:ve}
Let $\phi \in C^2(\R^N \setminus \{0\})$. Suppose that $\oz$ satisfies \eqref{volume} and \eqref{eqn:1as}. Then there are constants $C=C(\P, M_\phi, m_{\phi/\psi}, m_\psi)>0$, $r=r(\P, |\Omega_0|)>0$ and $h_0=h_0(\P, M_\phi, m_{\phi/\psi}, m_\psi, |\Omega_0|)$ such that the following holds for all $h\in (0,h_0)$ and $t\geq 0$:
\begin{equation*}
E_t(h) \hbox{ satisfies }\eqref{eqn:1as}, \quad  \B_r \subset E_t(h), \hbox{ and } \big| |E_t(h)| - |\oz| \big| < C \sqrt{h}.
\end{equation*}
\end{thm}

From the above theorem and Theorem~\ref{thm:bal}, it follows that $E_t(h)$ is uniformly bounded for all $t\geq 0$ and $0<h<h_0$. Based on this bound and Theorem~\ref{thm:lpb} the following is immediate:

\begin{cor}
\label{cor:lu}
Let $\phi$, $\oz$, $h_0$ and $r$ be as given in Theorem~\ref{thm:ve}. Then for $h\in (0,h_0)$ and $t\geq 0$, $E_t(h)$ is a uniformly bounded Lipschitz domain with a uniform Lipschitz constant. In particular, the cone property \eqref{eqn:lpb12} holds with $\Om = E_t(h)$ and $r = r(\P, |\Omega_0|)$.
\end{cor}

\subsection{Proof of Theorem~\ref{thm:ve}.}

 The main ingredient in the proof is the comparison of solutions with Wulff-shape self-similar barriers, based on the geometric properties obtained in Section~\ref{sec:pre}. 
 
 \medskip
 
 First we show 
a bound on the speed of the boundary of a solution that will be used for $E_t(h)$.

\begin{lem}
\label{lem:wulff-bound}
Suppose that $\Lambda \in C([0, h])$ with $\min \Lambda \geq \Lambda(0) -q$ for some  $q \geq 0$. Let $(\Omega_t)_{t \geq 0}$ be a viscosity solution of $V = \psi (\n)(-\kappa_{\phi} + \Lambda')$, and set $m=m_{\phi/\psi}$.  For given $x_0\in \R^N$, if we have
\begin{align}
\label{initial-wulff-bound}
x_0 + \left(2\delta + \dfrac{q}{m}\right) W_{\phi} \subset\subset \Omega_0  \quad\hbox{ for some } \delta>0,
\end{align}
then, for  $\tau_\delta:=\delta^2m/(N-1)$, 
\begin{align}
\label{wulff-bound}
x_0 + \delta W_{\phi} + \big(q + \Lambda(t) - \Lambda(0)\big) W_{\psi} \subset\subset \Omega_t \quad\text{ for } 0 \leq t \leq \min\{h, \tau_\delta\}.
\end{align}

A parallel statement holds for $\Omega_0$ and $\Omega_t$ replaced by $(\Omega_0)^\cc$ and $(\Omega_t)^\cc$, respectively, but with $-W_\psi$ replacing $W_\psi$. 
\end{lem}

\begin{proof}
We use the fact that the set flow $R(t) W_{\phi}$ is a viscosity subsolution of $V = - \psi(\n)\kappa_{\phi}$ with \[R(t) := \sqrt{4\delta^2 - 2(N-1)t / m_{\phi/\psi}}\] as long as $R(t)>0$, that is, $t < 2\tau_\delta$, since the set flow $R(m_{\phi/\psi}t)W_{\phi}$ is a solution of $V = -\phi(\n) \kappa_{\phi}$ and $\psi m_{\phi/\psi} \leq \phi$ (see \cite{Bel10}). It follows from Lemma~\ref{lem:conv} that the sup-convolution 
\begin{align*}
S(t) := x_0 + R(t) W_{\phi} + (q + \Lambda(t) - \Lambda(0)) W_{\psi}
\end{align*}
is a viscosity subsolution of $V = \psi(\n)(-\kappa_{\phi} + \Lambda')$ for $0 \leq t \leq \min\{h, 2\tau_\delta\}$.  Since $W_\psi \subset m_{\phi/\psi}^{-1}W_{\phi}$ (Lemma~\ref{lem:bw}),  our assumption yields $S(0) \subset\subset \Omega_0$.  Now we can conclude by the comparison principle  (Theorem~\ref{thm:comp}) since $R(t) \geq \delta$ for $0\leq t \leq 3\tau_\delta/2$.\end{proof}

We next show that $|E_h(t)|$ does not grow apart from $|\Omega_0|$ over more than one time step.

\begin{prop}
\label{prop:vol}
Suppose that $E$ satisfies \eqref{eqn:1as} and contains $\B_r$ given in \eqref{eqn:b} for some $r>0$. Then there exists $h_0 = h_0(\P, M_\phi, m_{\phi/\psi}, r)>0$ such that the following holds for $h \in (0,h_0)$:
\begin{align*}
\hbox{ if $|E| < |\oz|$, then $ E \subset \T_h(E;h)$} \,\,\hbox{ and }\,\, \hbox{ if $|E| > |\oz|$, then $\T_h(E;h) \subset E$. }
\end{align*}
\end{prop}

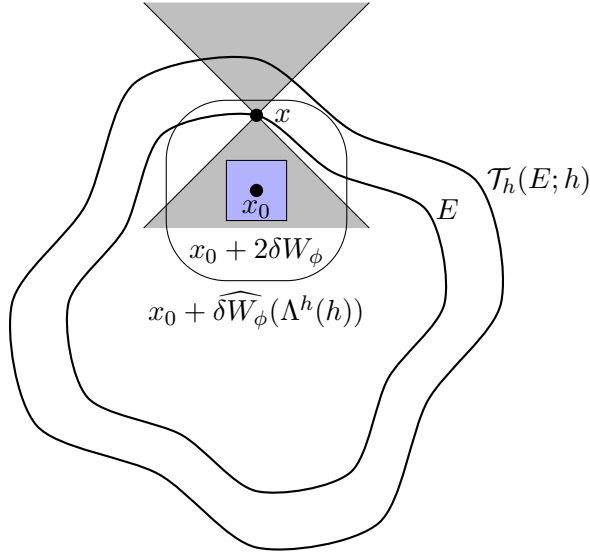
\begin{figure}
\centering
\begin{tikzpicture}[scale=2.5]
\draw[fill=gray!50!white] (-0.6,0.4) -- (0,1) -- (0.6,0.4);
\draw[fill=gray!50!white] (-0.6,1.6) -- (0,1) -- (0.6,1.6);

\def\curve{\draw[thick] plot [smooth cycle] coordinates {
(1,0) (0.9,0.5) (0.4,0.7) 
(0,1) (-0.5,0.9) (-0.7,0.4)
(-1,0) (-0.9,-0.5) (-0.4,-0.7)
(0,-1) (0.5,-0.9) (0.7,-0.4)};}

\curve
\draw (0.9,0.5) node[right] {$E$};
\begin{scope}[scale=1.3]
\curve
\draw (0.9,0.5) node[right] {$\mathcal T_h(E; h)$};
\end{scope}

\fill (0,1) node[right=3pt] {$x$} circle [radius=1pt];

\draw[fill=blue!30!white] (-0.16,0.76) rectangle +(0.32,-0.32);
\draw (0,0.4) node[below] {$x_0 + 2\delta W_{\phi}$};
\fill (0,0.6) node[below] {$x_0$} circle [radius=1pt];
\begin{scope}[yshift=0.6cm,scale=0.16]
\draw (3,1) arc [start angle=0,end angle=90,radius=2]
--(-1,3) arc [start angle=90,end angle=180,radius=2] 
--(-3,-1) arc [start angle=180,end angle=270,radius=2] 
-- node[below] {$x_0 + \widehat{\delta W_{\phi}}(\Lambda^h(h))$}
(1,-3) arc [start angle=270,end angle=360,radius=2] 
-- cycle;
\end{scope}
\end{tikzpicture}
\caption{Barrier argument in Proposition~\ref{prop:vol} if $|E| < |\oz|$}
\label{fig:vol}
\end{figure}

\begin{proof}
Let us assume that $|E|<|\oz|$ 
and fix $x \in E$.  By Theorem~\ref{thm:lpb}  $E$ contains $x-\co_r(A)$ given by some basis $A \subset \P$. By definition of $\sigma_3$ in \eqref{eqn:s3} this cone contains a ball of radius $s\sigma_3$ centered at $x_s := x - \frac{s}{2N} \sum_{p\in A} p$ for any $0 < s \leq r$, 
and since $W_{\phi} \subset \overline{B}_{M_\phi}(0)$ by Lemma~\ref{lem:bw}, we have
for $\delta_s:= s\sigma_3/(2M_\phi)$
\begin{align*}
x_s + 2\delta_s W_{\phi} \subset \overline B_{s \sigma_3}(x_s) \subset x - \co_r(A) \subset  E \quad \hbox{ for any } 0< s < r.
\end{align*}
For given $h > 0$ we set
$$
 s =s_h:= \frac{2M_\phi}{\sigma_3} \sqrt{\frac{N-1}{m_{\phi/\psi}}} \sqrt h.
$$
Let $h_0 > 0$ be the solution of $s_{h_0} = r$ and fix $h \in (0, h_0)$. 

Recall now that $\Omega_t:=\T_t(E; h)$ is by definition a viscosity solution of $V = \psi (\n)(-\kappa_\phi + (\Lambda^h)')$ with initial data $E$. Setting $\delta = \delta_s$ and $q = 0$, and noting that $\Lambda^h(t) \geq 0 = \Lambda^h(0)$  for $0\leq  t\leq h$ since $|E| < |\oz|$, Lemma~\ref{lem:wulff-bound} yields the last inclusion in
\begin{align}
\label{B-Lambdah}
x \in \overline B_{m_{\psi}\Lambda^h(h)}(x_s) \subset x_s +\delta_s W_{\phi} + \Lambda^h(h) W_{\psi}
\subset \T_h(E; h)
\end{align}
since $h = \tau_{\delta_s}$, where the first inclusion follows from $|x - x_s| \leq s/2 \leq m_\psi \alpha \sqrt{h} = m_\psi \Lambda^h(h)$ by definition of $\alpha$ in \eqref{eqn:al}, and the second one by Lemma~\ref{lem:bw}. We deduce $E \subset \T_h(E; h)$ since $x \in E$ was arbitrary.

\medskip

The inequality $|E| > |\Omega_0|$ can be handled similarly, considering the complements of the sets.
\end{proof}

Next we show that $\partial\T_t(E;h)$ stays in $O(\sqrt{h})$-neighborhood of $\partial E$.

\begin{prop}
\label{prop:cct} 
\begin{align*}
\T_t(E; h)^\cc \subset E^\cc +  \overline B_{c\sqrt h}(0) \quad \hbox{ and }  \quad \T_t(E;h) \subset E + \overline B_{c\sqrt h}(0) \,\, \hbox{ for all } t\in [0,h].
\end{align*}
Here $c := ( 2  + m_{\phi/\psi}^{-1} ) M_\phi \alpha$ with $m_{\phi/\psi}$, $M_\phi$ as given in \eqref{eqn:mphi}.
\end{prop}

\begin{proof}
To prove the first inclusion, suppose  $x_0 \notin E^\cc + \overline B_{c\sqrt h}(0)$. 
Since $\overline B_{c\sqrt h}(x_0) \subset E$, Lemma~\ref{lem:bw} implies that
\begin{align*}
x_0 + ( 2  + m_{\phi/\psi}^{-1} ) \alpha \sqrt{h} W_{\phi} \subset E.
\end{align*}
On the other hand, $\Lh$   satisfies $\Lambda^h(t) \geq \Lambda^h(0) - \alpha\sqrt{h}$ for  $0\leq t \leq h$.
Thus Lemma~\ref{lem:wulff-bound} with $\delta = q = \alpha \sqrt{h}$ yields
\begin{align*}
x_0 + \alpha \sqrt{h} W_{\phi} \subset \T_t(E;h) \hbox{ for all } 0 \leq t \leq \tau_\delta=h\alpha^2m_{\phi/\psi}/(N-1).
\end{align*}
As $\alpha$ satisfies \eqref{eqn:al}, $h\leq \tau_\delta$, and thus the first inclusion is obtained. Parallel arguments yield the other inclusion.
\end{proof}

\begin{prop}
\label{prop:err}
For given $r>0$, there exist positive constants $h_0$ and $C$ depending only on $\P, M_\phi, m_{\phi/\psi}, m_\psi, r$ such that the following holds for all $0<h<h_0$:

If a set $E$ satisfies \eqref{eqn:1as}, contains $\B_r$ in \eqref{eqn:b} and satisfies $\big| |E| - |\oz| \big| < C \sqrt{h}$, then 
 $$
\big| |\T_h(E;h)| - |\oz| \big| < C \sqrt{h}.
$$

\end{prop}

\begin{proof}
Proposition~\ref{prop:cct} and Proposition~\ref{prop:dis} yield
\begin{align}
\label{eqn:err12}
\big| |\T_t(E;h)| - |E| \big| < c \left( 2  + m_{\phi/\psi}^{-1} \right) M_\phi \alpha \sqrt{h} \hbox{ for all } t \in [0,h]
\end{align}
 for some $c = c(\P)$. If $|E| = |\oz|$, then we can conclude with 
\begin{align}
\label{eqn:c}
C(\P, M_\phi, m_{\phi/\psi}, m_\psi) := c \left( 2  + m_{\phi/\psi}^{-1} \right) M_\phi \alpha.
\end{align}

\medskip

If $|E|<|\oz|$, Proposition~\ref{prop:vol} yields $h_0$ such that $|E| \leq |\T_h(E;h)|$ for all $h \in (0,h_0)$. Thus we have
\begin{align*}
-C\sqrt{h}\leq |E| - |\oz| \leq |\T_h(E;h)| - |\oz| < |\T_h(E;h)| - |E| \hbox{ for all } h \in (0, h_0),
\end{align*}
where the first inequality is given by the assumption. Combining the above with \eqref{eqn:err12}, we conclude with the same choice of $C$. Parallel arguments work if $|E| > |\oz|$.
\end{proof}

\medskip

\textbf{Proof of Theorem~\ref{thm:ve}.} 

\medskip
Let $C=C(\P, M_\phi, m_{\phi/\psi}, m_\psi)$ be the constant given in \eqref{eqn:c}. Due to \eqref{volume}, we have
\begin{align}
\label{eqn:h}
|\oz| > |B_{\K \rho}(0)| + 2C \sqrt{h_0} + \e \qquad \hbox{ for } h_0, \e (= \e(\P, |\Omega_0|)) \ll 1,
\end{align}
where $\K$ is given in \eqref{eqn:k0}. In particular, as $|\oz| > |B_{\K \rho}(0)| + \e$, Theorem~\ref{thm:pre} yields $\B_r \subset \oz$  for some $r>0$ depending only on $\P$ and $\e = \e(\P, |\Omega_0|)$. 

\medskip

Now Proposition~\ref{prop:err} applies to $E = \oz$ and we have $\big| |E_{h}(h)| - |\oz| \big| < C \sqrt{h}$ for all $h \in (0,h_0)$, if $h_0$ is chosen sufficiently small. From this and \eqref{eqn:h}, we get
\begin{align*}
|E_{h}(h)| > |B_{\K \rho}(0)| + C\sqrt{h}+ \e \hbox{ for all } h \in (0,h_0).
\end{align*}
Again, as $|E_{h}(h)| > |B_{\K \rho}(0)| + \e$, Theorem~\ref{thm:pre} now yields that $E_h(h)$ satisfies \eqref{eqn:1as} and contains $\B_r$ for the same $r$ given above. 

\medskip

We can iterate this argument in the time interval $[kh,(k+1)h]$ for all $k \in \N$. Applying Proposition~\ref{prop:err} to $E = E_{kh}(h)$ and then using Theorem~\ref{thm:pre}, we conclude that
\begin{align}
\label{eqn:ve31}
E_{(k+1)h}(h) \hbox{ satisfies } \eqref{eqn:1as}, \hbox{ contains } \B_r \hbox{ and satisfies } \big| |E_{(k+1)h}(h)| - |\oz| \big| < C \sqrt{h}
\end{align}
for all $k \in \N$ and $h \in (0,h_0)$.

\medskip

Lastly, it follows from \eqref{eqn:err12} with $E = E_{\lfloor t/h\rfloor h}$ and \eqref{eqn:ve31} that
\begin{align*}
\big| |E_{t}(h)| - |\oz| \big| \leq \big| |E_{t}(h)| - |E_{\lfloor t/h\rfloor h}| \big| + \big| |E_{\lfloor t/h\rfloor h}| - |\oz| \big| < 2C\sqrt{h} \hbox{ for all } t \geq 0 \hbox{ and } h \in (0,h_0).
\end{align*}
The above and \eqref{eqn:h} imply that
\begin{align*}
|E_t(h)| > |B_{\K \rho}(0)| + \e \hbox{ for all } t \geq 0 \hbox{ and } h \in (0,h_0). 
\end{align*}
By Theorem~\ref{thm:pre} again, $E_t(h)$ satisfies \eqref{eqn:1as} and contains $\B_r$ for all $t\geq 0$ and $h \in (0,h_0)$.
\hfill$\Box$

\subsection{Equicontinuity of \texorpdfstring{$\Lh$}{Lamda h} and \texorpdfstring{$E_t(h)$}{Et(h)}}

\begin{prop}
\label{prop:hol}
Let $\oz$ and $\phi$ be as in Theorem~\ref{thm:ve}.
Then $\Lh: \Rpz \to \R$ in \eqref{eqn:et} is locally uniformly H\"{o}lder continuous. Namely there exist positive constants $h_1$ and $C$ depending only on $\P, m_\psi, M_\phi, m_{\phi/\psi}$ and $|\Omega_0|$  such that
\begin{align}
\label{eqn:1hol}
\left| \Lh(t) - \Lh(s) \right| \leq C |t-s|^{\frac{1}{2}}\quad \hbox{ for any } |t-s| \leq 1 \hbox{ and } h \in (0, h_1).
\end{align}
\end{prop}

\begin{proof}
The proof is based on the fact that $|E_t(h)|$ is  $O(h^{1/2})$-close to $|\Omega_0|$ (Theorem~\ref{thm:ve}), which bounds the oscillation range of $\Lambda^h$. 

We will only show \eqref{eqn:1hol} for small time intervals.  Let $\alpha$ satisfy \eqref{eqn:al} and let $r$ and $h_0$ be as given in Theorem~\ref{thm:ve}. For $\e_0 := r^2 / (2\alpha)^2$, we claim that there are constants $ h_1\in (0, h_0)$ and $C \geq \alpha$ with dependence as given above such that \eqref{eqn:1hol} holds for $|s - t| \leq \e_0$. 

It is clearly enough to show the claim for $s<t$. 
Suppose that for some $h \in (0, h_0)$, $s^*$, $t^*$ with $s^* < t^* \leq s^* +\e_0$ we have
\begin{align*}
R := \Lambda^h(t^*) - \Lambda^h(s^*) - \alpha |t^*-s^*|^{1/2} > 0.
\end{align*}
Let $s_0$ be the last time in $[s^*, t^*]$ such that $\Lambda^h(s_0) \leq \Lambda^h(s^*)$. Since $\Lambda^h$ is continuous $s_0$ is well-defined, less than $t^*$, and satisfies $\Lambda^h(\tau) > \Lambda^h(s_0) = \Lambda^h(s^*)$ for $\tau \in (s_0, t^*]$.   We first show that 
\begin{align}
\label{Es0-inclusion}
E_{s_0}(h) + \overline B_R(0) \subset E_t(h).
\end{align}

To deduce this, we apply Lemma~\ref{lem:wulff-bound} as in the proof of Proposition~\ref{prop:vol}, but this time with $s := 2\alpha|t-s_0|^{1/2} \leq 2\alpha \e_0^{1/2} \leq r$. Thus $\delta := s \sigma_3 / (2 M_\phi)$ satisfies $\tau_\delta = \delta^2 m_{\phi/\psi}/(N-1) > |t - s_0|$.
Using these estimates we can proceed as in the proof of Proposition~\ref{prop:vol} to show that for every $x \in E_{s_0}(h)$ we can find $x_0$ satisfying $x_0 + 2\delta W_{\phi} \subset E_{s_0}(h)$ and $|x - x_0| \leq s /2 = \alpha|t-s_0|^{1/2}$. Lemma~\ref{lem:wulff-bound} yields an inclusion as in \eqref{B-Lambdah} and we deduce
\begin{align*}
\overline B_R(x) \subset \overline B_{\Lambda^h(t) - \Lambda^h(s_0)}(x_0) \subset E_t(h).
\end{align*}
This implies \eqref{Es0-inclusion}. 

\medskip

We now use a series of inequalities to estimate $R$. The inequality \cite[Theorem~3]{Kra16} with dimensional constant $C_1$ is used first, then \eqref{Es0-inclusion} and the isoperimetric inequality that yield another dimensional constant $C_2$, and finally the estimate in Theorem~\ref{thm:ve} is used to find constants $h_1$ and $C_3$ depending only on $\P, M_\phi, m_{\phi/\psi}, m_\psi$ such that $|E_t(h) \setminus E_{s_0}(h)| \leq C_3h^{1/2}$ and $|\Omega_0|/2 \leq |E_t(h)|$ for $h\in (0, h_1)$. These all yield
\begin{align*}
R \leq C_1 \frac{|(E_{s_0}(h) + \overline B_R(0))\setminus E_{s_0}(h)|}{\per(E_{s_0}(h) + \overline B_R(0))}
\leq C_2 |E_t(h)|^{(1-N)/N} |E_t(h) \setminus E_{s_0}(h)| \leq C_4 |\Omega_0|^{(1-N)/N} h^{1/2}.
\end{align*}
for $h \in (0, h_1)$ with $C_4$  depending only on $\P, M_\phi, m_{\phi/\psi}, m_\psi$.  Setting $M := C_4|\Omega_0|^{(1-N)/N}$, we therefore have
\begin{align*}
\Lambda^h(t^*) - \Lambda^h(s^*) = \alpha|t^* - s^*|^{1/2} + R \leq \alpha|t^*-s^*|^{1/2} + Mh^{1/2}.
\end{align*}
Since we also have $\Lambda^h(t^*) - \Lambda^h(s^*) \leq \frac{\alpha}{h^{1/2}} |t ^*- s^*|$ by the definition of $\Lambda^h$, it follows that 
\begin{align*}
\Lambda^h(t^*) - \Lambda^h(s^*) \leq C |t^* - s^*|^{1/2} \hbox{ with } C=\max(2\alpha, M).
\end{align*}
This shows one direction in \eqref{eqn:1hol} for $|t-s| \leq \e_0$.
 We can similarly bound $\Lambda^h(t) - \Lambda^h(s)$ from below by considering the same argument for the complement $E_t(h)^\cc$. By making the constant $C$ larger if necessary, we deduce \eqref{eqn:1hol} for any $|t-s| \leq 1$.
\end{proof}

Proposition~\ref{prop:hol} yields the equicontinuity of $E_t(h)$ in the Hausdorff distance.

\begin{cor}
\label{prop:cet}
Under the assumptions and with $h_1$ in Proposition~\ref{prop:hol}, there exist $C$ depending only on $\P, M_\phi, m_{\phi/\psi}, m_\psi$ such that 
\begin{align}
\label{eqn:1cet}
d_H(\partial E_s(h), \partial E_t(h)) \leq C |t-s|^{\frac{1}{2}} \hbox{ for any } |t-s| \leq 1 \hbox{ and } h \in (0, h_1).
\end{align}
\end{cor}

\begin{proof} 
Using the H\"{o}lder continuity of $\Lh$ in Proposition~\ref{prop:hol}, one can use the barrier argument as in  the proof of Proposition~\ref{prop:cct} to conclude that  for all $0\leq s  <t <s+1$ and  $0<h<h_1$
\begin{align}
\label{eqn:cet11}
\I^-(h) \subset E_t(h) \subset \I^+(h), \quad \hbox{ where } \I^\pm(h) := \{ x \in \Rn : \sd(x, E_s(h)) < \pm C_1 |t-s|^{\frac{1}{2}} \},
\end{align}
for some $C_1$ depending on parameters $\P, M_\phi, m_{\phi/\psi}, m_\psi$. As a consequence, we have
\begin{align}
\label{eqn:cet21}
\sup_{x \in \partial E_t(h)} d(x, \partial E_s(h)) \leq C_1 |t-s|^{\frac{1}{2}} \hbox{ for all } h \in (0, h_1).
\end{align}

\medskip

From the Lipschitz continuity in Corollary~\ref{cor:lu}, there exists $C_2 = C_2(\P) > 1$ such that
\begin{align}
\label{eqn:cet30}
d(x, \partial \I^\pm(h) ) \leq C_1 C_2 |t-s|^{\frac{1}{2}} \quad \hbox{ for all } x \in \partial E_s(h) \hbox{ and } h \in (0, h_1) \hbox{ if } |t-s| \leq \e
\end{align}
for sufficiently small $\e=\e(\P) > 0$.  On the other hand, \eqref{eqn:cet11} implies $d(x, \partial E_t(h)) \leq d(x, \partial \I^\pm(h))$. From this and \eqref{eqn:cet30}, we have
\begin{align*}
\sup_{x \in \partial E_s(h)} d(x, \partial E_t(h)) \leq C_1 C_2 |t-s|^{\frac{1}{2}} \quad \hbox{ for all } h \in (0, h_1) \hbox{ if } |t-s| \leq \e.
\end{align*}
We can now conclude \eqref{eqn:1cet} using the above inequality and \eqref{eqn:cet21}, as well as the triangle inequality.
\end{proof}

\section{The proof of Theorem~\ref{THM:1}}
\label{sec:gexi}

In this section we prove Theorem~\ref{THM:1}. We construct the flow as the limit as $h\to 0$ of the approximate flow with the discrete-time forcing $\lambda_h$ introduced in Section~\ref{sec:lam}.
The main ingredients in the analysis are the cone property and the H\"{o}lder time-continuity of the approximate flow obtained in Section~\ref{sec:lam}. Due to these properties, we can rule out potential fattening of the sets as $h$ tends to zero. Recall that the smoothness of $\phi$ is necessary both for the reflection comparison principle and for the preservation of \eqref{eqn:1as} in Section~\ref{sec:pre}.

\medskip

We resort to the level set approach, which is more convenient for convergence arguments. Let us define
\begin{align}
\label{eqn:uh}
u^h(x,t) := \max \{ -\sd (x, E_t(h)) , a \} \qquad \hbox{for } (x,t) \in \Rn \times \Rpz,
\end{align}
where $E_t(h)$ is as given in Section~\ref{sec:lam}. The value of the constant $a<0$ is not important and is chosen only so that the solution is bounded.

\begin{lem}
\label{lem:uhc}
Let $\phi$ and $\oz$ be as in Theorem~\ref{THM:1}. Then along a subsequence $u^h$ locally uniformly converges to $u$ in $\Rn \times [0,\infty)$ as $h$ tends to zero.
\end{lem}

\begin{proof}
By Arzel\`{a}--Ascoli, it is enough to show that $\{u^h\}$ is uniformly bounded and equi-continuous for sufficiently small $h$. For the bound, the lower bound follows from its definition, and the upper bound is a consequence of Corollary~\ref{cor:lu}.

$u^h$ is $1$-Lipschitz in space for all $h > 0$ due to its definition, and furthermore
$$
| u^h(x,t) - u^h(x,s) | \leq | \sd(x, E_t(h)) - \sd(x, E_s(h)) | \hbox{ for all } s,t \in [0,\infty) \hbox{ with } |t-s| <1.
$$
As the evolving sets $\{E_t(h)\}_{t \in [0,T]}$ for $h \in (0, h_0)$ are uniformly H\"{o}lder continuous in the Hausdorff distance from Proposition~\ref{prop:cet}, we conclude. 
\end{proof}

\begin{prop}
\label{lem:cpt}
Let $\oz$ and $\phi$ be as given in Theorem~\ref{THM:1}. For $u$ given in Lemma~\ref{lem:uhc} and the corresponding subsequence $\{h_i\}_i$, we have 
\begin{align}
\label{eqn:1nof}
\limsups_{i \rightarrow \infty} \chi_{E_t(h_i)} = \chi_{\overline{\ot}}\,\, \hbox{ and } \,\, \liminfs_{i \rightarrow \infty} \chi_{E_t(h_i)} = \chi_{\ot} \hbox{ for each } t\geq 0,\end{align}
where
\begin{align*}
\ot := \{ u(\cdot,t) > 0\} \hbox{ and } \overline{\ot}^{\cc} = \{u(\cdot,t)<0\}. 
\end{align*}
\end{prop}

\begin{proof}
We only show the first equality in \eqref{eqn:1nof}, since the other can be shown by parallel arguments.  If  $u(x,t)>0$,  $u^{h_i}(x,t) > 0$ for sufficiently large $i$. Thus, if we define $\ot:= \{u(\cdot,t)>0\}$, then  
\begin{align*}
\limsups_{i \rightarrow \infty} \chi_{E_t({h_i})}(x,t) = 1 \hbox{ for all } x\in \overline{\ot}.
\end{align*}

Next we show that $u(\cdot, t)$ is negative outside of $\overline{\ot}$. Fix $x$ with $u(x,t) = 0$. By the uniform convergence $u^{h_i}$ to $u$ and the definition of $u^{h_i}$ as the signed distance function, there exists $y_i \to x$ with $u^{h_i}(y_i, t) = 0$. 
 Pick $x_0 \in \B_r^\cc \cap B_{r/2}(x)$ and let $A := \{p_j\}_{j=1}^N \subset \P$ be the basis from Theorem~\ref{thm:lpb} for $x_0$.  It follows from \eqref{eqn:lpb12} applied to $E_t(h_i)$ that 
\begin{equation}\label{cone2}
-\sd(\cdot, y_i - \co_r(A)) \leq u^{h_i}(\cdot, t) \leq \sd(\cdot, y_i +\co_r(A)) \qquad \hbox{on } B_{|a|}(y_i),
\end{equation}
where $a$ is the constant from \eqref{eqn:uh}.
Due to the local uniform convergence $u^{h_i} \to u$, \eqref{cone2} holds for $u$ too with $x$ in place of $y_i$, and we conclude that $x \in \overline{\{u(\cdot, t) >0\}} = \overline{\ot}$ and so $u(\cdot, t) < 0$ in $(\overline{\ot})^\cc$.

\medskip

Now, if $x\in (\overline{\ot})^{\cc}$, then $u(\cdot,t)<0$ in a compact neighborhood of $x$ and thus $u^{h_i}(\cdot,t)<0$ for sufficiently large $i$ in the same neighborhood. Thus we conclude that
\begin{align*}
\limsups_{i \rightarrow \infty} \chi_{E_t({h_i})}(x,t) = 0.
\end{align*}
\end{proof}

Let us finish this section with carefully verifying the statements of Theorem~\ref{THM:1}.

\medskip

{\bf Proof of Theorem~\ref{THM:1}.}

\medskip

Let $\ott$ be as given in Proposition~\ref{lem:cpt}. From Proposition~\ref{prop:hol}, along a subsequence $\Lh$ locally uniformly converges to $\Lam\in C^{1/2}([0,\infty))$ as $h \to 0$. We combine Lemma~\ref{lem:glue} with Proposition~\ref{lem:cpt} and Theorem~\ref{thm:sta} to conclude that $\ott$ is a viscosity solution of $V = \psi(\n)(-\kp + \Lam')$.

\medskip

As $E_t(h)$ satisfies \eqref{eqn:1as} and contains $\B_r$ due to Theorem~\ref{thm:ve}, Proposition~\ref{lem:cpt} yields the same property for  $\ot$.  Moreover  $\ot$ is uniformly bounded for all $t\geq 0$: this follows from Theorem~\ref{thm:bal} and the  fact that $\ot$ contains $\B_r$. Hence one can find a finite number of neighborhoods $\mathcal{O}_i:=B_r(x_i)$ with $x_i \in \B_r^\cc$, such that $\bigcup_i (\mathcal{O}_i \times [0,\infty))$ contains $\Gamma$. 

\medskip

As explained in the proof of Proposition~\ref{lem:cpt}, from the local uniform convergence of $u^h$ to $u$ and their non-degeneracy it follows that \eqref{eqn:lpb12} holds for $\ott$, where $A$ is the basis given from Theorem~\ref{thm:lpb}. In particular it follows that $\partial\ot$ has interior and exterior $r$-cone properties in each $B_r(x_i)$, with the axis of the cone $\nu_i$ only depending on  $x_i$. In other words $\partial\ot$ in each set $\mathcal{O}_i$ can be represented as a Lipschitz graph $\{x\cdot \nu_i = f_i(x',t)\}$, where $x' = x - (x\cdot\nu_i)\nu_i$, with $f_i(\cdot,t)$ having uniform Lipschitz constant over $t$.

\medskip

 The H\"{o}lder regularity of $f_i$ in time follows from
 $$
  d_H(\partial\ot, \partial\Omega_s) \leq C |t-s|^{\frac{1}{2}} \hbox{ for any } |t-s| \leq 1
  $$
 with $C$ only depending on $\phi$ and $\psi$. This is a consequence of Corollary~\ref{prop:cet} and the fact that $d_H(\partial\ot, \partial E_t(h))$ tends to zero as $h\to 0$, for each fixed $t\geq 0$.  We can now conclude.

\medskip

Lastly, it remains to show that $|\ot| = |\oz|$ for all $t \geq 0$. Since $\Omega_t$ is a Lipschitz domain, we have $|\ot| = |\overline{\ot}|$. Proposition~\ref{lem:cpt} and Fatou's lemma yield that
\begin{align*}
|\overline{\ot}| = \int_{\Rn} \limsups_{i \rightarrow \infty} \chi_{E_t(h_i)}(x) \;dx \geq \limsup_{i \rightarrow \infty} \int_{\Rn}  \chi_{E_t(h_i)}(x) \;dx = \limsup_{i \rightarrow \infty} |E_t(h_i)|
\end{align*}
and 
\begin{align*}
|\ot| = \int_{\Rn} \liminfs_{i \rightarrow \infty} \chi_{E_t(h_i)}(x) \;dx \leq \liminf_{i \rightarrow \infty} \int_{\Rn}  \chi_{E_t(h_i)}(x) \;dx = \liminf_{i \rightarrow \infty} |E_t(h_i)|.
\end{align*}
From Theorem~\ref{thm:ve} we have $\lim_{i \rightarrow \infty} |E_t(h_i)| = |\oz|$, and  thus we conclude. It follows that $(\ott, \Lam)$ is a viscosity solution of \eqref{model} in the sense of Definition~\ref{def:model}, and we can conclude the theorem.
\hfill$\Box$

\section{Global existence: Crystalline flows}
\label{sec:global-existence}

In this last section we focus on the remaining case of Lipschitz but non-differentiable anisotropy $\phi$. Given $\phi$, let $\{\phi_n\}$ be a sequence of positively one-homogeneous, convex functions in $C^2(\Rn\setminus\{0\})$ such  that $\phi_n\to\phi$ locally uniformly with
\begin{itemize}
\item $\{\phi_n < 1\}$ are strictly convex; 
\item all $\phi_n$ have a reflection symmetry with respect to the same root system $\mathcal P$ as $\phi$.
\end{itemize}
 For example, we can take positively one-homogeneous functions $\phi_n$ such that $\{\phi_n < 1\} = \{p :(\phi * \eta_{1/n})(p) + |p|^2 / n < 1\}$, where $\eta_{1/n}$ is the standard mollifier with radius $1/n$.
 
 \medskip
 
  We can then follow Section~\ref{sec:lam} to construct a sequence of approximate solutions for the anisotropy $\phi_n$ with forcing $\Lambda_n$. Using this approximation, we present two results which characterize our limit with the available notion of solutions of \eqref{model}.

\medskip

If $\phi \in C^2(\Rn \setminus \{0\})$, this would  immediately follow from the stability properties of viscosity solutions. When $\phi$ is only Lipschitz and its graph has corners, that is, more than one tangent hyperplane at a point, there are challenges associated to defining a notion of solutions, even with smooth forcing. As we already explained in Section~\ref{sec:intro}, the main challenge lies in the nonlocality of the evolution as the crystalline curvature depends on the size and shape of flat facets parallel to the facets of the Wulff shape, and is in fact very sensitive (discontinuous) to \emph{facet breaking} and \emph{facet bending}.  The crystalline curvature is usually constant on facets and they are therefore preserved in the evolution, but for certain geometries of facets or for non-uniform forcing the curvature might not be a constant and facets break or bend \cite{BNP99}. In dimension $N = 2$ the flow is relatively well understood and a notion of viscosity solutions has been defined \cite{GigaGiga98,GigaGiga01}. In dimensions $N \geq 3$ there are two notions of solutions only recently available. 

\medskip

In Section~\ref{sec:general-mobility} we discuss the notion of viscosity solutions  introduced by Giga and Pozar~\cite{GigPoz16,GigPoz18}, based on the level set functions. This notion allows a general mobility $\psi$ but requires the anisotropy $\phi$ to be purely crystalline. In Section~\ref{se:phi-regular-mobility} we discuss the alternative notion by Chambolle, Morini, Novaga and Ponsiglione~\cite{CMP17,CMNP19}. Here variational approach is used to introduce a notion of solutions, using the signed distance function to the evolving set. This notion directly deals with the set evolution and thus fits well with the approach taken in the preceding sections.  On the other hand, when the mobility $\psi$ is not {\it $\phi$-regular} such as $\psi \equiv 1$, a solution is indirectly defined as a limit of solutions with $\phi$-regular mobilities.

\subsection{General mobility for purely crystalline case}
\label{sec:general-mobility}

In this section we assume $\phi$ to be a {\it purely crystalline} anisotropy, that is, $\phi(p) = \max_{i} (x_i \cdot p)$, where $x_i$ are the vertices of the corresponding Wulff shape. This is so that we can use the notion of viscosity solutions developed in \cite{GigPoz16,GigPoz18,GigPoz20}.

\medskip

In contrast to Section~\ref{se:phi-regular-mobility}, we can consider a mobility $\psi$ that is not $\phi$-regular, as long as it is reflection symmetric with respect to the root system $\P$.  On the other hand,  in this generality the only stability result currently available in \cite{GigPoz20} is for limits of continuous solutions of the problem with smooth anisotropy. For this reason we follow Section~\ref{sec:lam} but \emph{without} the re-initialization of the distance function after every $h$-step to construct a sequence of continuous approximate solutions. This way we obtain approximating solutions for the level set equation. 

\medskip

 Due to issues with possible fattening and consequent non-uniqueness, it is not clear whether the zero super-level set of the limit solution has the correct volume. Thus our existence result for the set flow \eqref{model} is under the assumption of no fattening; see Theorem~\ref{th:crystalline-existence-no-fat}.   

\medskip

Let us fix $n \in \N$, $h = 1/n$ and $a < 0$ throughout this section. For given fixed $n \in \N$ the algorithm is parallel to that of Section~\ref{sec:lam} except the re-initialization step as we describe. With initial data $u_0$ in \eqref{freezing-init-data} and $\lambda_n(t) := \lambda(E, 1/n)$ for $t \in (0, h)$ where $\lambda$ is defined in \eqref{eqn:lam}, we find the unique continuous viscosity solution  $u_n \in C_a(\Rn \times [0, h])$ of the level set formulation of 
$$V = \psi(\n) (- \kappa_{\phi_n} + \lambda_n) \hbox{ in } \Rn \times [0, h], $$
 with initial data $u_0$. Then for every $k \in \N$, we iteratively define $\lambda_n(t) := \lambda(\{u_n(\cdot, kh) > 0\}, 1/n)$ for $t \in (kh, (k+1)h)$ and extend $u_n$ to $C_a(\Rn \times [0, (k+1)h])$ to be the unique continuous viscosity solution of the level set formulation of $V = \psi(\n)(- \kappa_{\phi_n} + \lambda_n)$ in $\Rn \times [kh, (k+1)h]$ with initial data $u_n(\cdot, kh)$ (that is, we do not reinitialize the data as the signed distance function to $\{u(\cdot, kh) > 0\}$).

\medskip

Note that 
\begin{equation}\label{um-Et}
\{u_n(\cdot, t) > 0\} = E_t(1/n)\qquad \text{and} \qquad \{u_n(\cdot, t) \geq 0\} = \overline{E_t(1/n)}.
\end{equation}
Indeed, on the interval $[0, h)$ the equality holds by the definition of $E_t(h)$ and no fattening is established in Corollary~\ref{cor:lu}. Furthermore, the absence of fattening also implies that reinitializing the signed distance function in \eqref{freezing-init-data} using $E = E_h(h)$ to continue the construction of $E_t(h)$ on the interval $[h, 2h)$ does not change the zero level set and therefore the equality continues to hold on this interval. Iteratively we conclude that the equality holds for all $t \geq 0$.   

Proposition~\ref{prop:hol} yields a constant $h_1 =h_1(\P, m_\psi, M_\phi, m_{\phi/\psi})>0$ such  that for $n >1/h_1$ we have uniform H\"{o}lder continuity for $\Lambda_n(t) := \int_0^t \lambda_n(s)\;ds$. Thus along a subsequence $\Lambda_n$ locally uniformly converges to $\Lambda\in C^{1/2}([0,\infty))$.

\medskip

By viscosity solutions in the sense of Definition~\ref{def:visl} below, we mean the generalization of viscosity solutions defined in \cite{GigPoz20} to a continuous $\Lambda$ as in Definition~\ref{def:visl},  using~\eqref{eqn:the2} in the sense of \cite{GigPoz20}.

\begin{prop}
\label{pr:crystalline-half-limits}
The functions
\begin{align*}
\overline u := \limsups_n u_n, \qquad \underline u := \liminfs_n u_n
\end{align*}
are a viscosity subsolution and a viscosity supersolution, respectively, of the crystalline mean curvature flow $V= \psi(\vec{n})(- \kappa_{\phi}+ \Lambda')$ in the sense of Definition~\ref{def:visl}.
\end{prop}

\begin{proof}
We modify the proof of \cite[Th.~8.9]{GigPoz16}. For $\delta > 0$ we can approximate every $\phi_n$ by the smooth $\phi_{n,\delta}(p) := (\phi_n * \eta_\delta) + \delta|p|^2$ and find viscosity solutions $u_{n,\delta}$ of 
\begin{align}
\label{unif-elliptic}
u_t = \psi(-D u) (-\dv D \phi_{n,\delta}(-D u) + \lambda_n)
\end{align}
with initial data $u(\cdot, 0) = u_0$ as above. By the stability result for the viscosity solutions of the smooth anisotropic mean curvature flow, $u_{n, \delta} \to u_n$ locally uniformly. Note that this approximation by solutions with uniformly elliptic operator $-\dv D \phi_{n,\delta}(-D u)$ is necessary in the proof of \cite[Th.~8.9]{GigPoz16} to be able to use the perturbed test function method.

Let us now take any $\Theta \in C([t_1, t_2]) \cap C^1((t_1, t_2))$ with $\Theta > \Lambda$ on $[t_1, t_2]$. For $n$ sufficiently large, $\Theta - \Lambda_n > 0$ on $[t_1, t_2]$ due  to the local uniform convergence of $\Lambda_n$ to  $\Lambda$. For these $n$, by definition and by following the proof of Lemma~\ref{lem:glue} the function $\hu_{n,\delta}(\cdot; \Theta - \Lambda_n)$ (see \eqref{eqn:sup} for the notation) is a viscosity subsolution of 
\begin{align*}
u_t =\psi(-Du)(- \dv D \phi_{n, \delta}(-D u) + \Theta')
\end{align*}
in $\Rn \times (t_1, t_2)$. 
We note that
\begin{align*}
\hu_{n, \delta}(\cdot, \Theta - \Lambda_n) \to \hu_n(\cdot, \Theta - \Lambda_n) \quad \text{as } \delta \to 0
\end{align*}
locally uniformly and, by Lemma~\ref{le:uhat-limsup},
\begin{align*}
\hat{\overline u}(\cdot; \Theta - \Lambda) = \limsups_n \hu_n(\cdot; \Theta - \Lambda_n).
\end{align*}

Therefore we can follow the proof of \cite[Th.~8.9]{GigPoz16}, with necessary modifications to allow a time-dependent forcing done in \cite{GigPoz20}, to conclude that $\hat{\overline u}(\cdot; \Theta - \Lambda) $ is a viscosity subsolution of \eqref{eqn:the2}.
\end{proof}

\begin{lem}
\label{le:crystalline-half-lim-conincidence}
$\overline u$ and $\underline u$ from Proposition~\ref{pr:crystalline-half-limits} satisfy $\overline u(\cdot, 0) \leq u_0 \leq \underline u(\cdot, 0)$. In particular, $\overline u = \underline u$ in $\Rn \times [0,\infty)$ and the subsequence $u_n$ converges to the continuous function $u := \overline u$ locally uniformly.
\end{lem}

\begin{proof}
We need to construct barriers at $t = 0$ for $u_n$ that can be controlled in $n$ uniformly.
Fix $\e > 0$ and choose $\Theta \in C^1([0, 1])$ with $\Lambda < \Theta < \Lambda + \e$ in $[0,1]$. As $\Lambda_n \to \Lambda$ locally uniformly (along a subsequence), for large enough $n$ we have $\Lambda_n < \Theta < \Lambda_n + \e$. We have that $\hu_n(\cdot; \Theta - \Lambda_n)$ is a subsolution of 
\begin{align}
\label{smooth-lset-pde}
u_t = \psi(-Du) (-\dv D\phi_n(-Du) + \Theta').
\end{align}

Let us recall the positively one-homogeneous level set function $\phi_n^\circ$ of the Wulff shape $W_{\phi_n} = \{\phi_n^\circ \leq 1\}$ defined as $\phi_n^\circ(x) := \max \{x \cdot p: \phi_n(p) \leq 1,\ p \in \Rn\}$. By a standard convex analysis, $\phi_n^\circ \in C^2(\Rn \setminus \{0\})$ and due to the local uniform convergence $\phi_n \to \phi$, we deduce $\phi_n^\circ \to \phi^\circ$ locally uniformly.
Let us define
\begin{align*}
v_n(x,t) := M_{\phi_n}\sqrt{\phi_n^\circ(-x)^2 + \tfrac{2(N-1) M_{\psi}}{m_{\phi_n}} t } + 
\tfrac {M_\psi M_{\phi_n}}{m_{\phi_n}} (\Theta(t) - \Theta(0)).
\end{align*}
We check that it is a classical supersolution on $(\Rn \setminus \{0\}) \times (0, \infty)$. Indeed, we have
\begin{align*}
\partial_t v_n(x,t) &= \tfrac{M_\psi M_{\phi_n}}{m_{\phi_n}} (\tfrac{N-1}{\sqrt{\cdots}} + \Theta'(t)),\\
Dv_n(x,t) &= M_{\phi_n} \tfrac{\phi_n^\circ(-x)}{\sqrt{\cdots}} (-D\phi_n^\circ(-x)).
\end{align*}
A standard convex analysis (see \cite{Bel10}) yields 
\begin{align*}
-\dv D_p\phi_n(-Dv_n(x,t)) = -\dv D_p\phi_n(D\phi_n^\circ(-x)) = \frac{N-1}{\phi_n^\circ(-x)}
\end{align*}
and $\phi_n(D\phi_n^\circ(-x)) = 1$. The latter yields $|D\phi_n^\circ(-x)| \leq \frac1{m_{\phi_n}}$ and, using that  $\tfrac{\phi_n^\circ(-x)}{\sqrt{\cdots}} \leq 1$, also $\psi(-Dv_n) \leq \tfrac{M_\psi M_{\phi_n}}{m_{\phi_n}}$. We conclude that
\begin{align*}
\psi(-Dv_n(x,t)) (-\dv D_p\phi_n(-Dv_n(x,t)) + \Theta'(t)) \leq \tfrac{M_\psi M_{\phi_n}}{m_{\phi_n}} (\tfrac{N-1}{\sqrt{\cdots}} + \Theta'(t)) = \partial_t v_n(x,t).
\end{align*}
Therefore $v_n$ is a classical supersolution of \eqref{smooth-lset-pde} in $(\Rn \setminus {0}) \times (0, \infty)$. Since the first term in $v_n$ is nondecreasing in $t$, we conclude that it is a viscosity supersolution in $\Rn \times (0, \infty)$.

Furthermore, as $\phi_n^\circ(-x) \geq M_{\phi_n}^{-1} |x|$, we observe that $v_n(x, 0) \geq |x|$ and therefore for any fixed $y \in \Rn$,
\begin{align*}
\hu_n(x, 0; \Theta(0) - \Lambda_n(0)) = \hu_0(x; \Theta(0) - \Lambda_n(0)) \leq v_n(x-y, 0) + \hu_0(y; \Theta(0) - \Lambda_n(0)), \qquad x \in \Rn,
\end{align*}
as both $u_0$ and $\hu_0$ are $1$-Lipschitz in space.
By comparison principle,
\begin{align*}
\hu_n (x,t)\leq v_n(x-y,t) + \hu_0(y; \Theta(0) - \Lambda_n(0))\qquad \hbox{ for }  (x,t) \in\Rn \times [0,1],
\end{align*}
where the right-hand side converges locally uniformly to $v(x-y, t) + \hu_0(y; \Theta(0) - \Lambda(0))$. Here $v$ is defined as $v_n$ but with $\phi$ instead of $\phi_n$.
We deduce
\begin{align*}
\limsups_n \hu_n(y, 0) &\leq v(y-y, 0) + \hu_0(y; \Theta(0) - \Lambda(0)) = \hu_0(y; \Theta(0) - \Lambda(0)) \\
&\leq u_0(y) + M_\psi|\Theta(0) - \Lambda(0)|\\
&\leq u_0(y) + M_\psi\e,
\end{align*}
where we used that $u_0$ is $1$-Lipshitz and that $W_\psi \subset \oB_{M_\psi}$ for the second inequality.
Since $\e >0$ was arbitrary, we conclude $\overline u(\cdot, 0) \leq u_0$. The inequality for $\underline u$ can be deduced analogously.

\medskip

Following the argument in \cite[Theorem 2.10]{KimKwo18b}, we can then show that $\overline u \leq \underline u$ in $\Rn \times [0, \infty)$ and hence $u := \overline u = \underline u$. We conclude that $u_n$ converges to $u$ locally uniformly along a subsequence.
\end{proof}

Now we turn to the question of the volume of the zero super-level set of the limit $u$. We will show that if the sets $\{u(\cdot, t) >0\}$ and $\{u(\cdot, t) \geq 0\}$ have the same volume, then it must be $|\Omega_0|$. To this end, we fix $a < 0$ and define the signed distance functions similar to \eqref{eqn:uh},
\begin{align*}
d_n(x,t) := \max(-\sd(x, \{u_n(\cdot, t) > 0\}), a).
\end{align*}
 Following the proof of Lemma~\ref{lem:uhc}, this sequence is locally uniformly bounded and equicontinuous. Thus by selecting a further subsequence if necessary,  there exists a continuous function $d$ such that $d_n \to d$ locally uniformly. 

\begin{lem}
\label{le:crystal-dist-u-relation}
Let $u$ be as in Lemma~\ref{le:crystalline-half-lim-conincidence} and $d$ be as introduced above. We have
\begin{align*}
\{u > 0\} \subset \{d > 0 \} \subset \{d \geq 0\} \subset \{u \geq 0\}.
\end{align*}
\end{lem}
\begin{proof}
Observe that  \eqref{um-Et} yields $\sign u_n = \sign d_n$. Let us fix $(x,t)$ with $u(x,t) > 0$. By continuity, there exists $\delta > 0$ with $\min_{\overline B_\delta(x)} u(\cdot, t) > 0$ and so by the locally uniform convergence we have $\min_{\overline B_\delta(x)} u_n(\cdot, t) > 0$ for sufficiently large $n$. In particular, $d_n(x,t) \geq \delta$ and hence $d(x,t) \geq \delta$.

A parallel argument verifies that $u(x,t) <0$ implies $d(x,t) < 0$. The claim of the theorem follows.
\end{proof}

\begin{thm}
\label{th:crystalline-existence-no-fat}
Let $\Omega_0$ satisfy~\eqref{eqn:1as} and let $u$ be from Lemma~\ref{le:crystalline-half-lim-conincidence}. Set $\Omega_t := \{u(\cdot, t) > 0\}$.
If there is no fattening of $\{u(\cdot, t)=0\}$ in measure, that is, if $|\{u(\cdot, t) = 0\}| = 0$ for all $t \geq 0$, then $(\ott, \Lam)$ is a viscosity solution of \eqref{model}.
\end{thm}

\begin{proof}
We can follow the proof of Theorem~\ref{THM:1} with $d_n$ in place of $u^h$ in \eqref{eqn:uh} to show that $\{d(\cdot, t) = 0\}$ can be locally expressed as a graph of a Lipschitz function and $|\{d(\cdot, t) > 0\}| = |\Omega_0|$. 

By Lemma~\ref{le:crystal-dist-u-relation}, for any $t \geq 0$
\begin{align*}
|\{u(\cdot, t) > 0\}| \leq |\{d(\cdot, t) > 0\}| \leq |\{u(\cdot, t) >0\}| +|\{u(\cdot, t) = 0\}| = |\{u(\cdot, t) > 0\}|.
\end{align*}
We conclude that $|\{u(\cdot, t) > 0\}| = |\Omega_0|$. Now the theorem follows from the previous characterization of $u$ as a viscosity solution of $V= \psi(\vec{n})(- \kappa_{\phi}+ \Lambda')$.
\end{proof}

\subsection{With \texorpdfstring{$\phi$}{phi}-regular mobility}
\label{se:phi-regular-mobility}

In this section we assume $\psi$ to be {\it $\phi$ regular}, namely that there exists $\e_0>0$ and a convex function $\eta$ such that 
\begin{equation}\label{uniform_phi_reg}
\psi(\nu) = \eta(\nu) + \e_0\phi(\nu).
\end{equation}  \eqref{uniform_phi_reg} is equivalent to ensuring that, for a closed set $E$, positive level sets of distance function $d:= \dist^{\psi^o}(x,E)$ satisfy the interior Wulff-shape property. When $\phi$ is differentiable, this property yields the curvature bound
\begin{equation}\label{bound_00}
(\dv z)_+ \leq  (N-1)/(\e_0d), \hbox{ where } z := D\phi(D d)
\end{equation}
(see \cite{CMNP19, chambolle_AP} for further discussions on $\phi$-regularity). Based on this observation, a notion of distributional solutions for the set evolution of $V = \psi(\n)(-\kappa_{\phi} + \lambda)$ was introduced in \cite{chambolle_AP} as well as its uniqueness, when $ \lambda$ is in $L^{\infty}_{loc}([0,\infty))$.

\medskip

We are not able to obtain such regularity for the volume-preserving forcing term $\Lambda'$ for our limit flow, which only exists in the distributional sense. (In general it appears difficult to obtain strong regularity properties for $\Lambda$ in non-convex setting: see \cite[Example A.2]{KimKwo18b}). Instead, here we will show that our limit satisfies a natural extension of the distributional solutions in \cite{CMNP19}, with necessary modifications to address the weaker regularity of our forcing term. We expect this notion to deliver uniqueness for crystalline flows of this form (that is with fixed forcing $\Lambda'$ where its anti-derivative $\Lambda$ is merely in $C([0,T])$),  however we do not pursue this issue here.

\medskip

For $\phi_n$ as given earlier in the section, we define $\psi_n(\nu):= \eta(\nu)+\e_0\phi_n(\nu)$. Let us denote its corresponding solution $((\Omega_t^n)_{t\geq0}, \Lambda_n)$ and define
$$
d_n(x,t):= \dist^{\psi_n^o} (x,\Omega_t^n) \hbox{ and } \tilde{d}_n(x,t):=\dist^{\psi_n^o} (x,(\Omega_t^n)^\cc),\,\,\, \hbox{ where } \dist^\eta(x,E):= \inf_{y\in E} \eta(x-y).
$$
Note that due to the geometric properties we have on $(\Omega_t^n)_{t\geq0}$ and the uniform H\"{o}lder continuity of $\Lam_n$ (see Theorem~\ref{THM:1}), along a subsequence, $\Omega^n_t$ converges to $\Omega_t$ locally uniformly in Hausdorff distance, and $d_n, \tilde{d}_n, \Lam_n$ converge to $d, \tilde{d}, \Lam$  locally uniformly in space and time. Below we will show that the limiting flow $(\ott, \Lam)$ satisfies the properties of distributional solutions for the crystalline flow.

\begin{thm}\label{sec7.1_main}
Let $(\psi_n, \phi_n)$ be as given above, and  let $(\ott, \Lam)$ be a subsequential limit of $((\Omega_t^n)_{t\geq0}, \Lambda_n)$ as discussed above. Then the following holds for $E_t:=\overline{\Omega}_t$:
\begin{itemize}
\item[(a)] Let $d(x,t):= \dist^{\psi^o}(x,E_t)$. Then there exists $z\in L^{\infty}(\Rn \times (0,T))$ such that 
$z\in \partial\phi(D d)$ a.e., $\dv z$ is a Radon measure in $\Sigma:=\bigcup_{0<t<T}(\Rn  \setminus E_t) \times\{t\}$, and
\begin{equation}\label{upper_bound}
(\dv z)^+ \in L^{\infty}(\{d(x,t) \geq \delta\}) \hbox{ for every } \delta \in (0,1).
\end{equation} 
Moreover, for any smooth $\phi$ supported in $\Sigma$ we have 
\begin{equation}\label{level_set_pde}
\int \int d(-\phi_t)  dx dt \geq \int\int (z\cdot D \phi - \Lambda \phi_t) dxdt.
\end{equation} 
\item[(b)] The statements of $(a)$ hold for $E_t$ replaced by $\tilde{E}_t:= \Rn \setminus \Omega_t$ and $\Lambda$ replaced by $-\Lambda$.\\
\item[(c)] $|\Omega_t| = |\Omega_0|$ for all  $t>0$.
\end{itemize}
\end{thm}

\begin{remark}
In Lemma 2.6 of \cite{chambolle_AP}, there is an additional term $-Md \phi$ in the right-hand side integrand of \eqref{level_set_pde}. This term is present due to the spatial dependence of the forcing $\Lambda$ in their case,  and thus does not appear for our problem.
\end{remark}
 
The following is an immediate consequence of properties $(a)$--$(b)$, which constitutes of the definition in \cite{CMP17} for the flow with the fixed forcing $V = \psi(\n)(-\kappa_\phi +g)$ with $g = \Lambda'$. 

\begin{cor}
When $\Lambda$ is a Lipschitz continuous function of time, $\ott$ is a solution of the flow $V = \psi(\n)(-\kappa_\phi + \Lambda')$ in the sense of \cite{CMP17}.
\end{cor}

 Our proof largely follows that of Theorem 2.8 in \cite{chambolle_AP}, with necessary modifications made for the low regularity of $\Lambda$. We will only show $(a)$ since $(b)$ can be shown via a parallel proof. $(c)$  is a direct consequence of the following convergence: 
 $$
\sup_{x\in \Omega^n_t}  d(x, \Omega_t) \to 0 \hbox{ as } n\to \infty, \hbox{ for all } t>0.
 $$

\medskip

\textbf{Proof of Theorem~\ref{sec7.1_main}.}
Let us consider a sequence of $C^1$ functions $\theta_n$ that sits below $\Lambda_n$ and locally uniformly converges to $\Lambda$.  For instance we can choose $\tilde{\theta}_n:= \Lambda*  \eta_{1/n}$ with a standard mollifier $\eta$ and shift it down by $e_n:=\|\tilde{\theta}_n - \Lambda_n\|_{L^\infty}$ to define $\theta_n$ (Note that $e_n$ goes to zero as $n\to\infty$ due to the locally uniform convergence of $\Lambda$ to $\Lambda_n$).  By Definition~\ref{def:visl}
 $$
 \tilde{u}_n(x,t):= \inf_{x+(\Lambda_n-\theta_n)(t)W_\psi}u(\cdot,t) = \chi_{\tilde{\Omega}^n_t}(x), \quad \hbox{ where } u(\cdot,t):= \chi_{\Omega_t},
 $$
 is a viscosity supersolution of $V = \psi(\n)(-\kappa_{\phi} + \theta_n')$. 
  We accordingly define 
  $$\tilde{d}_n:= \dist^{\psi^o}(x, E^n_t), \hbox{ where }E^n_t:= \overline{\tilde{\Omega}^n_t}.
  $$

\medskip

 From Lemma 2.6 of \cite{chambolle_AP}, $d_n$ satisfies \eqref{level_set_pde} with $z$ replaced by $z_n = D\phi_n(D d_n)$ and $\Lambda$ replaced by $\theta_n$. From Theorem~\ref{THM:1} which provides uniform spatial geometric properties on $\Omega^n_t$ and uniform H\"{o}lder continuity of $d_n$ over time, we know that $E^n_t$ converges to $E_t$ in the Kuratowski sense,  and thus $d_n$ locally uniformly converges to $d$ in $\Rn\times [0,\infty)$.  Moreover, $z_n$'s are uniformly bounded in $L^{\infty}_{loc} (\Rn\times [0,\infty))$ and thus have a subsequential weak-$* $ limit $z$. Using this and the locally uniform convergence of $\tilde{d}_n, \theta_n$ to $d, \Lambda$, we can confirm that \eqref{level_set_pde} holds for $d$ and $z$.

\medskip

Now it remains to confirm that $z\in \partial\phi(D d)$ a.e. with \eqref{upper_bound}.
To this end, observe that due to the uniform $\phi_n$ regularity of $\psi_n$ we have $\dv z_n \leq (N-1)/(\e_0d_n)$ for some $\e_0>0$ as pointed out in \eqref{bound_00}.  Hence arguing as in Theorem 2.8 of \cite{chambolle_AP} we can conclude.

\hfill$\Box$

\appendix

\section{Geometric properties}


Here we show several geometric properties used in the paper.  First we show that $\P_\phi$ given below is a root system:
\begin{align*}
\P_\phi := \{p \in \S^{N-1} : \phi = \phi \circ \Psi_p \hbox{ and } \psi = \psi \circ \Psi_p \}. 
\end{align*} 
By definition of $\P_\phi$ and the fact that $\Psi_p = \Psi_{-p}$, it can be shown that
\begin{align}
\label{eqn:sym2}
p \in \P_\phi, \quad \hbox{ if and only if } -p \in \P_\phi
\end{align}

Recall the reflection with respect to a hyperplane containing the origin
\begin{align*}
\Psi_{p} = \Psi_{\Pi_p(0)} = I - 2 p \otimes p
\end{align*}
is a symmetric unitary operator and an involution. Furthermore, compositions of three (or any odd number of) reflections are also reflections. From this observation, we show that if $p$ and $q$ are directions of reflection symmetry, then $\pm \Psi_{q}(p)$ is also a direction of reflection symmetry.

\begin{lem}
\label{lem:new}
If $p, q \in \P_\phi$, then $\pm \Psi_{q}(p) \in \P_\phi$. In particular, $\P_\phi$ is a root system.
\end{lem}

\begin{proof}
As $\Psi_{q}$ is an involution and symmetric, we have
\begin{align*}
\Psi_{q} \Psi_{p} \Psi_{q} x &= \Psi_{q} \left( \Psi_{q} x - 2 ( \Psi_{q} x \cdot p) p \right) = x - 2 (x \cdot \Psi_{q} p) \Psi_{q} p = \Psi_{\Psi_{q}p} x.
\end{align*}
From $| \Psi_{q}(p) | = | p | = 1$ and \eqref{eqn:sym2}, we conclude that $\pm \Psi_{p}(q) \in \P_\phi$.
\end{proof}


\begin{lem}
\label{lem:per}
The perimeter of a set $E$ satisfying \eqref{eqn:rp} and $\B_r \subset E \subset B_R(0)$ is bounded by $C = C(\phi, r, R)>0$.
\end{lem}
\begin{proof}
Set $F:= B_R(0) \setminus \B_r$. There exists a finite number of points $x_i, 1 \leq i \leq m$ in $F$ such that
\begin{align*}
F \subset \bigcup_{1 \leq i \leq m} B_r(x_i).
\end{align*}

As $E$ is a Lipschitz domain from Theorem~\ref{thm:lpb}, it suffices to show that $\H^{N-1}(\partial E \cap B_r(x_i))$ is uniformly bounded for $1 \leq i \leq m$. Here, $\H^{N-1}$ is the $(N-1)$-dimensional Hausdorff measure. Either $\partial E \cap B_r(x_i)$ is empty or it can be represented by a Lipschitz graph. In particular, from the cone condition in Theorem~\ref{thm:lpb}, the Lipschitz constant only depends on $r$ and $\phi$ and thus we conclude. 
\end{proof}

Next, let us recall the uniform density from \cite[Definition 4]{Kra16}. Let $c \in (0,1)$ and $s_0 > 0$. We say that $\Om \subset \Rn$ has $(s_0, c)$-uniform lower density if the estimate
\begin{align*}
0 < c \leq \frac{|B_s(x) \cap \Om|}{|B_s(x)|}
\end{align*}
holds for all $ s \in (0, s_0)$ and $x \in \partial \Om$. Similarly, $\Om$ is said to have $(s_0, c)$-uniform upper density if
\begin{align*}
\frac{|B_s(x) \cap \Om|}{|B_s(x)|} \leq 1 - c <1.
\end{align*}
When both conditions are satisfied together, $\Om$  has $(s_0, c)$-uniform density.

\begin{lem}\cite[Theorem 4]{Kra16}
\label{lem:kra}
Let $\Om \subset \Rn$ have $(s_0, c)$-uniform density. Then
\begin{align*}
| \{ x \in \Rn : 0 < d(x, \Om) < s \} | \leq C \left(1 + \tfrac{1}{c} \right)^{\frac{N-1}{N}} \per(\Om) s \hbox{ for all } s \in (0, s_0).
\end{align*}
Here, $C$ is a dimensional constant.
\end{lem}

As a consequence of Theorem~\ref{thm:lpb} and the above lemma, we conclude.

\begin{prop}
\label{prop:dis}
Suppose that $E$ satisfies \eqref{eqn:rp} and $E$ contains $\B_r$ for some $r>0$. Then, there exists $c = c(\P)>0$ such that 
\begin{align}
\label{eqn:dis}
| \{ x \in \Rn : 0 < d(x, E) < s \} | \leq c s
\end{align}
for all $s \in (0,r)$. Here, $\B_r$ is given in \eqref{eqn:b}
\end{prop}

\begin{proof}
We claim that $E$ has $(r, \sigma_3^N)$-uniform density for $\sigma_3$ given in \eqref{eqn:s3}. For all $s \in (0,r)$, $E$ has an $s$-interior cone and an $s$-exterior cone from Theorem~\ref{thm:lpb}. As $s$-interior and exterior cones are contained in a ball of radius $s$ and contains a ball of radius $\sigma_3 s$ for $\sigma_3$ given in \eqref{eqn:s3}, we conclude that
\begin{align*}
\sigma_3^N \leq \frac{|B_s(x) \cap E|}{|B_s(x)|} \leq 1 - \sigma_3^N.
\end{align*}
Then, we apply Lemma~\ref{lem:kra} and Lemma~\ref{lem:per} to conclude \eqref{eqn:dis}.
\end{proof}


A function $f : \Rn \to \R$ is called \emph{positively one-homogeneous} if
\begin{align}
\label{eqn:one-homogeneous}
f(s\xi) = s f(\xi) \qquad \text{for all } \xi \in \R^N \text{ and } s \geq 0.
\end{align}

Recall the definition of the Wulff shape $W_f$ in \eqref{Wulff}.
\begin{lem}
\label{lem:bw}
For positively one-homogeneous functions $f, g : \Rn \to \R$ with $f \leq g$ we have
\begin{align*}
\oB_{m_f}(0) \subset W_f \subset W_g \subset \oB_{M_g}(0),
\end{align*}
Here, $m_f$, $M_g$ are given in \eqref{eqn:mphi}.
\end{lem}

\begin{proof}
$W_f \subset W_g$ is clear from the definition. For the ordering with $\oB_{m_f}(0)$ and $\oB_{M_g}(0)$, we note that $\oB_r(0) = W_h$ for $h(p) := r|p|$ for any $r \geq 0$.
\end{proof}

\section{Technical lemmas}

\begin{lem}
\label{le:uhat-limsup}
Suppose that $u_k: \Rn\times \R \to \R$ is a locally-bounded sequence of upper semi-continuous functions and let $u := \limsups_k u_k$. Let $r_k \in C(\R)$ be a sequence of non-negative continuous functions such that $r_k \to r$ locally uniformly. Then
\begin{align*}
\sup_{\overline B_{r(t)}(x)} u(\cdot, t) = \limsups_{\substack{k\to\infty\\(\xi,s) \to (x,t)}} \sup_{\overline B_{r_k(s)(\xi)}} u_k(\cdot, s).
\end{align*}
\begin{align*}
\hat u(x, t; r) = \limsup_{\substack{k\to\infty\\(x_k,t_k) \to (x,t)}} \hat u_k(x_k, t_k; r_k)
\end{align*}
\end{lem}

\begin{proof}
Fix $(x,t)$ and $y \in \overline B_{r(t)}(x)$ such that $u(y, t) = \sup_{\overline B_{r(t)}(x)} u(\cdot, t) =: \hat u(x, t;r)$.

Suppose $(x_k, t_k) \to (x,t)$. We again fix $y_k \in \overline B_{r_k(t_k)}(x_k)$ such that $u_k(y_k, t_k) = \sup_{\overline B_{r(t_k)}(x_k)} u(\cdot, t_k) =: \hat u_k(x_k, t_k; r_k)$.

Consider a subsequence $k_m$ so that $\lim_m u_{k_m}(y_{k_m}, t_{k_m}) = \limsup_k u_k(y_k, t_k)$. Selecting a further subsequence (not relabeled), we may assume that $y_{k_m} \to z$. We have
\begin{align*}
|z - x| \leq |z - y_{k_m}| + |y_{k_m} - x_{k_m}| + |x_{k_m} - x| \leq |z - y_{k_m}| + r_{k_m}(t_{k_m}) + |x_{k_m} - x| \to r(t).
\end{align*}
This implies that $u(z, t) \leq u(y, t)$. Since $u = \limsups u_k$, we conclude that 
\begin{align*}
\limsup_k \hat u_k(x_k, t_k; r_k) = \limsup_k u_k(y_k, t_k) = \lim_m u_{k_m}(y_{k_m}, t_{k_m}) \leq u(z, t) \leq u(y, t).
\end{align*}
Since the sequence $\{(x_k, t_k)\}$ was arbitrary, we conclude that
\[
\hat u(x, t; r) \geq \limsup_{\substack{k\to\infty\\(x_k,t_k) \to (x,t)}} \hat u_k(x_k, t_k; r_k).
\]

To show the equality, we consider a maximizing sequence, i.e., we choose $(y_k, t_k) \to (y, t)$ such that $u(y,t) = \limsup_k u_k(y_k, t_k)$. We can also take a sequence $\{x_k\}$ such that $y_k \in \overline B_{r_k(t_k)}(x_k)$ and $x_k \to x$. Indeed, take $x_k = s_kx +(1- s_k)y_k$, where
\begin{align*}
s_k := \min(1, r_k(t_k) / |y_k - x|).
\end{align*}
We then have
\begin{align*}
\limsup_k \hat u_k(x_k, t_k; r_k) \geq \limsup u_k(y_k, t_k) = u(y, t) = \hat u(x, t; r).
\end{align*}
\end{proof}

\begin{lem}
\label{le:wulff-sup-conv}
Suppose that $\psi, \phi: \Rn \to [0, \infty)$ are positively one-homogeneous convex functions, with zero only at $p=0$ and suppose that $\phi \in C^2(\Rn \setminus \{0\})$.
Suppose that $u$ is a viscosity subsolution of $u_t = \psi(-Du) (-\dv D\phi(-Du) + \lambda)$ for some $\lambda \in C(\R)$. Then for any positive $R \in C^1(\R)$, 
$\hu(\cdot; R)$ from \eqref{eqn:sup} is a viscosity subsolution of $u_t = \psi(-Du) (-\dv D\phi(-Du) + \lambda + R')$.
\end{lem}

\begin{proof}
Without loss of generality we may assume that $u$ is upper semi-continuous.
To simplify the notation we write $\hu(x, t)$ instead of $\hu(x, t; R)$.
Let $\varphi$ be a smooth test function such that $\hu - \varphi$ has a maximum $0$ at $(\hat x, \hat t)$. Recall that we need to show $\varphi_t \leq F^*(\hat t, D\varphi, D^2\varphi)$ at $(\hat x, \hat t)$, where $F(t, p, X) := \psi(-p)\left(\trace [D_p^2\phi(-p) X]  + \lambda + R'\right)$, $p \neq 0$.

\medskip

Due to the assumption we have 
$$\varphi(x,t)\geq \hu(x,t)= \max_{x - R(t)W_\psi} u(\cdot, t)$$
with equality at $(\hat x, \hat t)$. We now fix $\hat y \in \hat x -  R(\hat t) W_\psi$ such that $u(\hat y, \hat t) = \hu(\hat x, \hat t)$. 
Note that from the definition $R(t) \tfrac{\hat x - \hat y}{R(\hat t)} \in R(t) W_\psi$ and so $x - R(t) \tfrac{\hat x - \hat y}{R(\hat t)} \in x - R(t) W_\psi$, which yields
\begin{align*}
\varphi(x,t) \geq \hu(x,t) \geq u(x - R(t) \tfrac{\hat x - \hat y}{R(\hat t)}, t)
\end{align*}
for all $x, t$ with equality at $(\hat x, \hat t)$. Thus we deduce
\begin{align*}
\hat \varphi(x, t) := \varphi(x + R(t) \tfrac{\hat x - \hat y}{R(\hat t)},t) \geq  u(x, t) 
\end{align*}
for all $x$, $t$ with equality at $(\hat y, \hat t)$. In particular, $u - \hat \varphi$ has a local maximum at $(\hat y, \hat t)$.

\medskip

Now a direct computation yields $D\hat \varphi(\hat y, \hat t) = D\varphi(\hat x, \hat t)$, $D^2 \hat \varphi(\hat y, \hat t) = D^2 \varphi(\hat x, \hat t)$ and
\begin{align}
\label{time-der-rel}
\hat \varphi_t (\hat y, \hat t)= \varphi_t (\hat x, \hat t)+ D\varphi (\hat x, \hat t)\cdot \tfrac{\hat x - \hat y}{R(\hat t)} R'(\hat t),
\end{align}

If $D\varphi(\hat x, \hat t ) = 0$ this simply yields $\hat \varphi_t(\hat y, \hat t)= \varphi_t(\hat x, \hat t)$ and we conclude that the correct viscosity solution condition is satisfied for $\varphi$ since $u$ is a viscosity solution with right-hand side $F(t, p, X) - \psi(-p) R'(t)$ by assumption.

Now suppose that $D\varphi(\hat x, \hat t ) \neq 0$.
As 
\begin{align*}
\varphi(x, \hat t) \geq \hu(x, \hat t) \geq u(\hat y, \hat t) = \varphi(\hat x, \hat t) \qquad \text{for }x \in \hat y + R(\hat t) W_\psi,
\end{align*}
we deduce that $ \hat y + R(\hat t)W_{\psi} \subset \{\varphi(\cdot, \hat t) \geq \varphi(\hat x, \hat t)\}$.
In particular, $-D\varphi(\hat x, \hat t)$ is an outer normal to $W_\psi$ at $(\hat x - \hat y)/ R(\hat t)$.
Therefore, by definition of $W_\psi = \{x : x\cdot p \leq \psi(p) \ \forall p\}$ and the fact that $\psi$ is positively one-homogeneous and convex,  
\begin{align*}
-D\varphi(\hat x, \hat t) \cdot \tfrac{\hat x - \hat y}{R(\hat t)} = \psi(-D\varphi(\hat x, \hat t)).
\end{align*}
which yields together with \eqref{time-der-rel} 
\begin{align*}
\hat \varphi_t (\hat y, \hat t)= \varphi_t (\hat x, \hat t)- \psi(-D\varphi(\hat x, \hat t)) R'(\hat t),
\end{align*}
again yielding the correct viscosity condition for $\varphi$ at $(\hat x, \hat t)$ from the viscosity solution condition that $\hat \varphi$ satisfies at $(\hat y, \hat t)$.

\medskip

We conclude that $\hu$ is a viscosity subsolution of $u_t = \psi(-Du) (-\dv D\phi(-Du) + \lambda + R')$.
\end{proof}

\bibliographystyle{alpha}
\bibliography{paper_ACMCF}

\end{document}